\documentclass[reqno,11pt]{amsart}

\usepackage{graphicx} % Required for inserting images
\usepackage{todonotes}
\usepackage{aliascnt} % Pacchetto fondamentale per l'aliasing
\usepackage[english]{babel}
\usepackage{hyperref}
\usepackage{amssymb,amsmath,amsthm,color,mathtools,tikz,subfig,dirtytalk,csquotes,soul,orcidlink,enumitem}
\usepackage[dvipsnames]{xcolor}
\usepackage[capitalise, nameinlink]{cleveref}
\usepackage{comment}
\usepackage{amsmath}
\usepackage{amsthm}
\usepackage[left=2.8cm,right=2.8cm,top=2.8cm,bottom=2.8cm]{geometry}

\numberwithin{equation}{section}

\theoremstyle{definition} % Esempi di solito sono in testo dritto
\newtheorem{example}{Example}[section]
\crefname{example}{Example}{Examples}

\theoremstyle{plain}

\newaliascnt{theo}{example}       % Crea un clone del contatore example
\newtheorem{theo}[theo]{Theorem}  % Definisce l'ambiente sul clone
\aliascntresetthe{theo}           % Sincronizza il reset
\crefname{theo}{Theorem}{Theorems} % Istruisce cleveref

\newaliascnt{prop}{example}
\newtheorem{prop}[prop]{Proposition}
\aliascntresetthe{prop}
\crefname{prop}{Proposition}{Propositions}

\newaliascnt{lemma}{example}
\newtheorem{lemma}[lemma]{Lemma}
\aliascntresetthe{lemma}
\crefname{lemma}{Lemma}{lemmata}

\newaliascnt{cor}{example}
\newtheorem{cor}[cor]{Corollary}
\aliascntresetthe{cor}
\crefname{cor}{Corollary}{Corollaries}

\theoremstyle{definition}

\newaliascnt{dhef}{example}
\newtheorem{dhef}[dhef]{Definition}
\aliascntresetthe{dhef}
\crefname{dhef}{Definition}{Definitions}

\newaliascnt{hyp}{example}
\newtheorem{hyp}[hyp]{Assumption}
\aliascntresetthe{hyp}
\crefname{hyp}{Assumption}{Assumptions}

\theoremstyle{remark}

\newaliascnt{remark}{example}
\newtheorem{remark}[remark]{Remark}
\aliascntresetthe{remark}
\crefname{remark}{Remark}{Remarks}

\def\elle#1{L^{#1}}
\def\elleom#1{L^{#1}(\Omega)}

\def\Sd{{\mathbb{S}^{d-1}}}
\renewcommand{\phi}{\varphi}

\def\C{\mathcal{C}}
\def\sub{\subseteq}

\newlist{assumption}{enumerate}{1}
\setlist[assumption,1]{label=(\roman*),ref=(\roman*)}
\crefname{assumptioni}{assumption}{assumptions}

\def\Haus{{\mathcal{H}^{d-1}}}

\def\Qom{{Q_\Omega}}

\def\O{\Omega}
\def\o{\omega}
\def\io{\int_\Omega}
\def\norma#1#2{ \|#1 \|_{#2}}
\def \e{\varepsilon}

\def\semin#1#2{ [#1 ]_{#2}}

\def\L{\mathcal{L}}

\def\N{\mathbb{N}}

\def\R{\mathbb{R}}
\def\Rd{\mathbb{R}^d}

\newcommand{\abs}[1]{\left|#1\right|}

\DeclareMathOperator{\supp}{supp}

\title{Closing the gap: Maz'ya-Shaposhnikova and asymptotics of fractional perimeters}

\author{Elisa Davoli
\orcidlink{0000-0002-1715-5004}}\email{elisa.davoli@tuwien.ac.at}
\address{TU Wien,
Institut für Analysis und Scientific Computing,
Wiedner Hauptstraße 8-10,
1040 Vienna, Austria}
\author{Alberto Fanizza
\orcidlink{0009-0007-8135-0537}}\email{alberto.fanizza@asc.tuwien.ac.at}
\address{TU Wien,
Institut für Analysis und Scientific Computing,
Wiedner Hauptstraße 8-10,
1040 Vienna, Austria}
\author{Marco Picerni
\orcidlink{0009-0004-4364-4831}}\email{mpicerni@sissa.it}
\address{SISSA, via Bonomea 265, 34136, Trieste, Italy}

\subjclass[2020]{46E35, 26A33, 46B20}

\keywords{Maz'ya-Shaposhnikova formula, Gagliardo seminorms, asymptotic volume ratio, $s$-perimeters, Gamma-convergence, metric measure spaces}

\begin{document}

% \begin{abstract}
%     We generalize the Maz'ya-Shaposhnikova formula to functions which are only locally integrable and, thus, not necessarily vanishing at infinity. By introducing a notion of mass at infinity, we identify an explicit expression for the asymptotics of the Gagliardo seminorm as the fractional parameter tends to zero. On the one-hand, our limiting functional reduces to the classical Lebesgue norm for functions with vanishing mass at infinity. On the other hand, it coincides with the pointwise limit previously identified for $s$-fractional perimeters when evaluated on characteristic functions. We further show that the same functional encodes the asymptotic behavior of Gagliardo seminorms in the sense of Gamma-convergence and provide an extension to the metric measure setting.
% \end{abstract}

\begin{abstract}
    We prove a generalization of the Maz'ya-Shaposhnikova formula in the case $p=2$ for functions that may not belong to ${L^2}(\mathbb{R}^d)$ and, thus, might not vanish at infinity.
    By introducing a notion of \emph{mass at infinity}, we explicitly characterize the limit as $s\to0^+$ of Gagliardo seminorms localized on a bounded Lipschitz domain $\Omega$. By `localized', we mean here that we account only for interactions involving at least one point in $\Omega$. The identified limiting functional  provides a unifying framework to link the classical Maz'ya-Shaposhnikova formula and the asymptotics of nonlocal perimeters. On the one hand, it reduces to the classical $L^2$ norm for functions that are globally integrable on $\mathbb{R}^d$. On the other hand, it recovers the pointwise limit of $s$-fractional perimeters when evaluated on characteristic functions of sets. 
    We further show that the same functional encodes the asymptotic behavior of Gagliardo seminorms in the sense of Gamma-convergence with respect to the weak-$L^2$ topology. Finally, we provide an extension to the setting of metric measure spaces.
\end{abstract}

\maketitle
\section{Introduction}

{

The study of the asymptotics of the Gagliardo seminorms is a central topic in nonlocal analysis, since it allows to provide a quantitative meaning to the idea that the Sobolev-Slobodeckij space $W^{s,p}$ should be, in some sense, close to $\elle p$ when $s$ is close to zero and close to $W^{1,p}$ for $s$ close to $1$. 

The first result in this direction is the Bourgain-Brezis-Mironescu (BBM) formula \cite[Theorem 2]{bourgain:hal-00747692}, which deals with the case $s\to1^-$. The formula states that for every function $u\in \elle p(\O)$, with $\O\sub\R^d$ bounded smooth domain, and for every $p\in(1,+\infty)$,
\begin{equation}\label{eq:BBM}
    \lim_{s\to1^-}(1-s)\iint_{\O\times \O}\frac{\abs{u(x)-u(y)}^p}{\abs{x-y}^{d+sp}}dxdy=K_{d,p}\norma{\nabla u}{\elle p(\O)}^p,
\end{equation}
with the convention that the right hand side equals $+\infty$ if $u\notin W^{1,p}(\O)$. Here, the constant $K_{d,p}$ depends only on $p$ and $d$. An analogous characterization for $p=1$ was also proved in \cite{bourgain:hal-00747692}, but only under the \textit{a priori} assumption that $u\in W^{1,1}(\O)$.  This result was successfully extended to functions in $BV(\O)$ by Dávila \cite[Theorem 1]{davila2002open} (see also the development by Leoni and Spector \cite{Leoni_Spector}). An analysis of the asymptotics in the sense of $\Gamma$-convergence was carried out by Ponce in \cite{ponce2004new}.

Recently, the validity of the BBM formula has been investigated beyond the Euclidean setting, encompassing, for example, metric measure spaces satisfying a Poincar\'e inequality \cite{DiMarino_Squassina,oleinik2025,oleinik2025asymptotic}, thin films \cite{Braides_BBM_thin}, as well as general anisotropic convolution kernels \cite{gennaioli2025sharp,Davoli_BBM_sharp,Davoli_BBM_antisym,Kubin_Saracco_stefani}.

The Maz'ya-Shaposhnikova (MS) formula \cite{MazyaShaposhnikova2002} provides a counterpart to \eqref{eq:BBM} for the asymptotics of the $s$-Gagliardo seminorm as $s$ approaches zero.
To be precise, given a function $u\in W^{s_0,p}(\R^d)$ for some $s_0$ and $p\in(1,+\infty)$, we have
\begin{equation}
\label{eq:intro:classicalMS}
    \lim_{s\to0^+} \frac{s}{2}\iint_{\R^d\times \R^d} \frac{\abs{u(x)-u(y)}^p}{\abs{x-y}^{d+sp}}dxdy=\frac{d\omega_d}{p}\norma{u}{\elle p(\R^d)}^p,
\end{equation}
where $\omega_d$ is the Lebesgue measure of the unit ball in $\R^d$. We stress that \eqref{eq:intro:classicalMS} only holds if $u\in L^p(\R^d)$, and so only if $u$ vanishes at infinity in a measure theoretic sense. The multiplicative factor $d\o_d$, which is the $d-1$ Hausdorff measure of the unit sphere $\Sd$, accounts for the fact that, as $s$ approaches zero, the fractional kernels
\[
\rho_s(x,y)=\frac{s}{\abs{x-y}^{d+sp}}
\]
concentrate their mass at infinity (in a measure-theoretic sense). 
Note, also, that $\omega_d$ encodes the asymptotic volume ratio (henceforth abbreviated $AVR$) of the metric measure space $(\Rd,d_{E},\mathcal{L}^d)$ with respect to the standard volume function $V(t)=t^d$.
The role played by the $AVR$ was made much more evident by the works \cite{Han-MS&BBM,HanPinamontiXu2025,HanPinamontirigidity}, where the MS formula was generalized to the metric measure setting. Notable examples of spaces for which the formula holds are finite-dimensional (non-Euclidean) Banach spaces, Carnot Groups and $MCP(K,N)$ spaces (which are metric spaces satisfying a synthetic upper bound on the dimension and a lower bound on the curvature, see \cite{Ohta2007MCP,Sturm2006II}).

The choice of restricting the analysis in the classical MS formula to functions with global $L^p$-integrability on the whole $\R^d$ comes with a structural limitation: forcing a function to vanish at infinity artificially suppresses its tail behavior, which the MS formula should intuitively emphasize. Roughly speaking, as $s\to0^+$, the kernel $s\abs{x-y}^{-(d+sp)}$ decays very slowly at the infinity, causing long-range interactions to account for most of the energy. 
%\todo[inline]{E: sposterei qui}
In fact, as $s\to0^+$, the tail of the function $\abs{\cdot}^{-(d+sp)}$ approaches the threshold of non integrability, given by the exponent $-d$. From the point of view of the Gagliardo seminorm, this means that the interactions that prevail in this regime are the ones between points that lie \say{at infinite distance}. This suggests that, in the limit, we should see an interaction energy between points of $\R^d$ and points \say{at infinity}. As an intuition, one can think that the fractional kernels force one of the two Lebesgue measures appearing in the integral \eqref{eq:intro:classicalMS} to become a $d-1$-Hausdorff measure concentrated on a sphere representing all the possible (oriented) directions of $\R^d$. In the classical Maz'ya-Shaposhnikova formula, this phenomenon is completely hidden and the interactions at infinity disappear within the $\elle p $ norm. 

%\todo[inline]{E: fine parte che cambierei}
%Forcing a function to vanish at infinity artificially suppresses precisely the asymptotic behavior we want to describe.

To properly overcome this limitation, i.e., to treat functions that may not vanish at infinity (and therefore fully capturing the highly nonlocal nature of the small-$s$ regime), we consider the \emph{Gagliardo seminorm with interior-exterior interactions:}
\begin{equation}
    \label{def:intro:gagliardoqomega}
    \semin{u}{W^{s,p}(\Qom)}^p=\iint_\Qom\frac{\abs{u(x)-u(y)}^p}{\abs{x-y}^{d+sp}}dxdy.
\end{equation}
Here, $\Omega$ is a bounded, Lipschitz domain and $\Qom=(\Rd\times\Rd)\setminus(\O^c\times\O^c)$. 
We say that a measurable function $u:\Rd\to\R$ belongs to $W^{s,p}(\Qom)$ for some $p\in[1,+\infty)$ if it belongs to $\elleom p$ and $\semin{u}{W^{s,p}(\Qom)}<+\infty$. Note that  $W^{s,p}(\Qom)\subset \elle p_{\text{loc}}(\Rd)$. In fact, since $\O$ is bounded, there exists a constant $C=C(\O,s,d)$ such that, for every $y\in\O$ and $x\in\R^d$,
\begin{equation*}
    \frac{1}{|x-y|^{d+sp}}\geq C\frac{1}{1+|x|^{d+sp}}.
\end{equation*}
Let now $u\in W^{s,p}(Q_\O)$ and assume, without loss of generality, that $u$ has zero mean in $\O$. Then, by Jensen's inequality,
\begin{align*}
    \semin{u}{W^{s,p}(\Qom)}^p&\geq C \int_{\R^d} \frac{1}{1+|x|^{d+sp}}\int_\O|u(x)-u(y)|^p dydx\\
    &\geq C \L^d(\O)\int_{\R^d}\frac{|u(x)|^p}{1+|x|^{d+sp}}dx.
\end{align*}
This extension of the Gagliardo seminorm, as evident from the decomposition
\begin{equation}
    \semin{u}{W^{s,p}(\Qom)}^p=\semin{u}{W^{s,p}(\O)}^p+2\int_\O\int_{\O^c}\frac{\abs{u(x)-u(y)}^p}{\abs{x-y}^{d+sp}}dxdy,
\end{equation}
isolates the energy arising from interactions involving at least one point in $\Omega$ and ignores all the exterior-exterior contributions. Far from being a mere technical variant, this is the functional used to model nonlocal phenomena on bounded domains. Indeed, it is the natural energy for minimization problems with Dirichlet boundary conditions (where the value of the function needs to be prescribed on the whole $\Omega^c$, rather than just on $\partial\Omega$) \cite{DKP_Nonlocal_harnack,Franzina-Palatucci-eigenvalues,Palatucci-sp-Dirichlet}. It also appears as singular perturbation in phase-transition models \cite{SavinValdinoci_gammaconv}, and in the theory of nonlocal minimal surfaces \cite{DSV16,DSV17,DSV20a,DSV20b,DSV26}. In particular %(and perhaps most importantly for this paper), 
it allows to define the fractional perimeter:
\[
\operatorname{Per}_s(E;\Omega)=\frac12\semin{\chi_E}{W^{s,1}(\Qom)},
\]
which approximates the De Giorgi perimeter functional as $s\to1^-$ \cite{ADpM-Perimeters,Fanizza-Perimeter}.\\

In the same spirit of the Maz'ya-Shaposhnikova formula, the asymptotic behavior of the fractional perimeter as $s\to0^+$ has been investigated by Dipierro, Figalli, Palatucci and Valdinoci in \cite{Quartet}. To be precise, they proved that, for any measurable set $E$ such that the quantity (also known as nonlocal tail \cite{KKP_note,KKP_Holdercont})
\begin{equation}
    \label{eq:alpha1}
\alpha_1(E)=\lim_{s\to 0^+}s\int_{B_1^c(0)}\frac{\chi_E(y)}{\abs{y}^{d+s}}dy
\end{equation}
is well-defined (that is, such that the limit in \eqref{eq:alpha1} exists), one has
\begin{equation}\label{eq:intro:quartet}
    s\operatorname{Per}_s(E;\Omega)=(d\omega_d-\alpha_1(E))\L^d(E\cap\O)+\alpha_1(E)\L^d(\O\setminus E).
\end{equation}
We stress that $\alpha_1(E)=0$ for every bounded set $E$. This implies that \eqref{eq:intro:classicalMS} and \eqref{eq:intro:quartet} give the same result when a bounded set is considered. We also recall that, as shown in \cite{Quartet}, the object on the right-hand side of \eqref{eq:intro:quartet} does not define a measure.

The result in \cite{Quartet} confirms the previously described intuition that  whenever a function $u$ has nonvanishing mass at infinity (in particular, $u$ may not be integrable on the whole $\R^d$), we should not expect $s\semin{u}{W^{s,p}(\Qom)}^p$ to converge to some multiple of $\norma{u}{\elleom p}^p$, but rather that its tail contributions should play a role in the asymptotics as $s\to 0^+$.

In this paper, we make the above intuition rigorous and prove an extension of the Maz'ya-Shaposhnikova formula to functions which may have nonzero mass at infinity. We focus here on the case $p=2$. For the case  of a general $p\geq 1$, a different approach is required; this is the content of the companion paper \cite{DaFaPi2}. 

To state our main result, we need to introduce the notion of $p$-mass at infinity.

\begin{dhef}
    \label{Def:alphau}
    Let $p\in \N$, and let $u:\R^d\to\R$ be a measurable function. We define the \emph{$p$-mass at infinity of $u$} as
    \begin{equation}
        \alpha_p(u):=\lim_{s\to0^+}s\int_{B_1^c(0)}\frac{u(y)}{\abs{y}^{d+sp}}\,dy,
    \end{equation}
    whenever this quantity is well--defined. 
\end{dhef}

Note that $\alpha_p(u)$ exists for some $p\geq 1$ if and only if $\alpha_q(u)$ exists for every $q\geq 1$ and, in that case 
\begin{equation}\label{eq:alpha1_alphap}
    \alpha_p(u)=\frac1p\alpha_1(u)
\end{equation}
for every such $u$ and $p\geq 1$.
To simplify notation, since  in the present paper we focus on the hilbertian case, we write $\alpha(u)$ instead of $\alpha_2(u)$. The notation $\alpha_p(u)\in \R$ will implicitly mean that $\alpha_p(u)$ exists and is finite; in this case, we say that $u$ has finite mass at infinity.

\begin{remark}
    \label{rk:alpha1}
    A direct computation in polar coordinates shows that, for every $p\geq 1$, $\alpha_p(1)=\frac{d\omega_d}{p}$.
\end{remark}

In the rest of the paper we will denote with $u^+=\max\{u,0\}$ and $u^-=\max\{-u,0\}$, respectively, the positive and negative part of $u$. We can now state the main result.

\begin{theo}
    \label{Theo:MS pointwise:euclidean}
    Let $u\in L^2_{loc}(\R^d)$ be such that $u\in H^{s_0}(\Qom)$ for some ${s_0\in(0,1)}$.
    Assume that $\alpha(u^+), \alpha(u^-), \alpha(u^2)\in \R$. 
    Then,
    \begin{equation}
        \label{Eq:intro:ExtMS:pw:eucl}
        \lim_{s\to 0^+}\frac{s}{2}[u]^2_{H^s(\Qom)}=\sum_{k=0}^{2}\binom{2}{k}(-1)^k\alpha(u^k)\int_\Omega u^{2-k}(x)\,dx.
    \end{equation}
\end{theo}

The key novelty of Theorem \ref{Theo:MS pointwise:euclidean} is that it not only holds for functions belonging to $\elle2(\R^d)$, for which it reduces to \eqref{eq:intro:classicalMS}, but also for functions in $\elle2_{\textrm{loc}}(\R^d)$ (such as periodic functions), which may not be integrable and thus exhibit a different behaviour as $|x|\to \infty$. 
Moreover, an application of Theorem \ref{Theo:MS pointwise:euclidean} to $u=\chi_E$ (assuming that $\alpha_1(E)$ exists), leads to
\[
\lim_{s\to 0^+}\frac{s}{2}[u]^2_{H^s(\Qom)}=\alpha(1)\L^d(E\cap\O)-2\alpha(E)\L^d(E\cap\O)+\alpha(E)\L^d(\O),
\]
which is (by \eqref{eq:alpha1_alphap} and a change of variable on the parameter $s$ in the left-hand side)
\begin{equation}
    \label{eq:intro:fracper:zero}
    \lim_{s\to 0^+}\frac s2\operatorname{Per}_s(E;\Omega)=\frac12\bigg((d\omega_d-\alpha_1(E))\L^d(E\cap\O)+\alpha_1(E)\L^d(\O\setminus E)\bigg).
\end{equation}
In this sense, Theorem \ref{Theo:MS pointwise:euclidean} generalizes the asymptotics for the fractional perimeter proved in \cite[Theorem 2.5]{Quartet} to a wider class of functions.\par In Theorem \ref{Theo:MS:integrale doppio}, we complete our pointwise analysis by further showing that for functions $u(r,\theta)$ (where $r>0$ and $\theta\in\Sd$) converging to $u_\infty(\theta)$ as $r\to+\infty$, the limit obtained in Theorem \ref{Theo:MS pointwise:euclidean} can be written in the compact form
\[
\frac12\io\int_\Sd\abs{u(x)-u_\infty(\theta)}^2dxd\Haus(\theta),
\]
thus providing a true interaction energy between the values of $u$ in $\O$ and at infinity.\\

\par
Our second result (see Theorem \ref{Theo:MSGamma}) is to show that the limit in \eqref{Eq:intro:ExtMS:pw:eucl} also holds in the sense of Gamma-convergence with respect to the weak-$L^2$ topology. As a direct consequence, we deduce that the functional appearing in \eqref{Eq:intro:ExtMS:pw:eucl} provides also the Gamma-limit as $s\to0^+$ of the fractional perimeters and it is therefore the lower semicontinuous envelope of the set functional appearing in \eqref{eq:intro:fracper:zero} with respect to the same topology .\par

Our third contribution consists in an extension of Theorem \ref{Theo:MS pointwise:euclidean} to the setting of metric measure spaces (see Theorem \ref{Theo:MS pointwise:mms}). We conclude our study with an application to Carnot Groups and a comparison with the results obtained for the Gaussian fractional perimeter \cite{GaussianFracPer}.\\
}

The paper is structured as follows: in \Cref{Sec:puntuale Rd} we build the technical infrastructure needed to prove Theorem \ref{Theo:MS pointwise:euclidean}, discuss possible extensions to the case $p\neq 2$, and prove Theorem \ref{Theo:MS:integrale doppio}. 
 Section \ref{Sec:Gammaconv} is devoted to the proof of Theorem \ref{Theo:MSGamma}, whereas in \Cref{Sec:puntuale mms} we extend \Cref{Theo:MS pointwise:euclidean} to the setting of metric measure spaces.

\section{The extended Maz'ya-Shaposhnikova formula in the Euclidean setting}\label{Sec:puntuale Rd}

In this section, we prove Theorem \ref{Theo:MS pointwise:euclidean} and we discuss possible extensions to the case $p\neq2$. Here and in what follows, $H^s$ will denote the space $W^{s,2}$.

\subsection{Mass at infinity and Asymptotic Volume Ratio}

To streamline the proof of Theorem \ref{Theo:MS pointwise:euclidean}, we first prove some preliminary results regarding the $p$-mass at infinity and its connection with the Asymptotic Volume Ratio of the Euclidean space equipped with a weighted measure. This establishes a connection with the setting of \cite{HanPinamontiXu2025}.

 The following lemma shows that the mass at infinity $\alpha(u)$ can be equivalently defined by considering balls $B_R(x)$ centered at any $x \in \mathbb{R}^d$, up to taking the limit as $R\to\infty$. This implies that the quantity $\alpha(u)$, when it exists, is invariant under translations of $u$, thus clarifying the terminology in Definition \ref{Def:alphau},

\begin{lemma}
    \label{lemma:lualphau}
    Let $u\in L^1_{loc}(\R^d)$. Assume that $\alpha(u^+),\,\alpha(u^-)\in \R$. Then,
    \begin{equation}
        \label{luwelldefined}
        \begin{split}
             L_u&:=\lim_{R\to\infty}\liminf_{s\to 0^+}s\int_{B_R^c(x)}\frac{u(y)}{\abs{x-y}^{d+2s}}\,dy\\
             &=\lim_{R\to\infty}\limsup_{s\to 0^+}s\int_{B_R^c(x)}\frac{u(y)}{\abs{x-y}^{d+2s}}\,dy
        \end{split}
    \end{equation}
    is well defined and does not depend on $x\in\R^d$. Moreover, $L_u=\alpha(u)$.
\end{lemma}
\begin{proof}
    We prove that, if $\alpha(u^+)$ and $\alpha(u^-)$ exist and are finite, then for any $x\in\R^d$ both limits in \ref{luwelldefined} exist and are equal to $\alpha(u)$.\par
    \textit{Step $1$.} Let $0<r<R<+\infty$. For any $x\in \R^d$ we have
    \begin{equation*}
        \limsup_{s\to 0^+} \left|s\int_{B_r^c(x)}\frac{u(y)}{\abs{x-y}^{d+2s}}dy-s\int_{B_R^c(x)}\frac{u(y)}{\abs{x-y}^{d+2s}}dy\right|=0.
    \end{equation*}
    In fact:
    \begin{equation*}
        \begin{split}
            \left|s\int_{B_r^c(x)}\frac{u(y)}{\abs{x-y}^{d+2s}}dy-s\int_{B_R^c(x)}\frac{u(y)}{\abs{x-y}^{d+2s}}dy\right|&\leq s\int_{B_r^c(x)\setminus B_R^c(x)}\frac{\abs{u(y)}}{\abs{x-y}^{d+2s}}\,dy\\
            &\leq \frac{s}{r^{d+2s}}\int_{B_r^c(x)\setminus B_R^c(x)}\abs{u(y)}\,dy,
        \end{split}
    \end{equation*}
    which converges to zero as $s\to0^+$ since $u\in L^1_{loc}(\R^d)$. In particular, choosing $x=0$, by the hypothesis on the existence of $\alpha(u^+)$ and $\alpha(u^-)$, we infer that $\alpha(|u|)$ exists as well, and for every $R>0$:
    \begin{equation}
        \label{Eq:alpha1R}
        \alpha(|u|)=\lim_{s\to0^+}{s}\int_{B_1^c(0)}\frac{|u(y)|}{\abs{y}^{d+2s}}\,dy
        =
        \lim_{s\to0^+}{s}\int_{B_R^c(0)}\frac{|u(y)|}{\abs{y}^{d+2s}}\,dy.
    \end{equation}
    
    \textit{Step 2.} We prove that for every $x\in\R^d$, there holds
    \begin{equation}
    \label{claim:step2}
    \lim_{R\to+\infty}\limsup_{s\to 0^+}   \left|s\int_{B_R^c(x)}\frac{u(y)}{\abs{x-y}^{d+2s}}\,dy-s\int_{B_R^c(0)}\frac{u(y)}{\abs{y}^{d+2s}}\,dy\right|=0.
    \end{equation}
    Indeed, we estimate
     \begin{multline*}
          \left|s\int_{B_R^c(x)}\frac{u(y)}{\abs{x-y}^{d+2s}}\,dy-s\int_{B_R^c(0)}\frac{u(y)}{\abs{y}^{d+2s}}\,dy\right|\\ \leq \left|s\int_{B_R^c(x)}\frac{u(y)}{\abs{x-y}^{d+2s}}\,dy-s\int_{B_R^c(0)}\frac{u(y)}{\abs{x-y}^{d+2s}}\,dy\right|\\
            +\left|s\int_{B_R^c(0)}\frac{u(y)}{\abs{x-y}^{d+2s}}\,dy-s\int_{B_R^c(0)}\frac{u(y)}{\abs{y}^{d+2s}}\,dy\right|
            :=I+II.
    \end{multline*}
    We proceed by estimating each summand separately.\\
    For the first term, we have that
    \begin{equation*}
        \begin{split}
            I&=\left|s\int_{B_R^c(x)}\frac{u(y)}{\abs{x-y}^{d+2s}}\,dy-s\int_{B_R^c(0)}\frac{u(y)}{\abs{y-x}^{d+2s}}\,dy\right|\\
            &  \leq s\int_{B_R(x)\Delta B_R(0)}\frac{\abs{u(y)}}{\abs{x-y}^{d+2s}}dy
        \end{split}
    \end{equation*}
    For $R>0$ big enough (depending on $x$), there exist $0<\delta_1<\delta _2$, (depending on $R$), such that $B_R(x)\Delta B_R(0)\subset B_{\delta_2}(0)\cap B_{\delta_1}^c(x)$. 
    
    Hence,
    \begin{equation*}
       I\leq s\int_{B_{\delta_2}(0)\cap B_{\delta_1}^c(x)}\frac{\abs{u(y)}}{\abs{x-y}^{d+2s}}dy \leq \frac{s}{\delta_1^{d+2s}}\int_{B_{\delta_2(0)}}\abs{u(y)}\,dy,
    \end{equation*}
    where the latter quantity converges to zero as $s\to 0^+$.
    For the second term, we rewrite
    \begin{equation*}
        II=\left|s\int_{B_R^c(0)}u(y)\left(\frac{1}{\abs{y-x}^{d+2s}}-\frac{1}{\abs{y}^{d+2s}}\right)dy\right|.
    \end{equation*}
    Let us fix $y\in \R^d$ and call $h_s(x)=\abs{y-x}^{-(d+2s)}$ for $x\neq y$. A direct computation yields
    \begin{equation*}
        \abs{Dh_s(x)}=\frac{(d+2s)}{\abs{y-x}^{d+2s+1}}.
    \end{equation*}
    For $R>2|x|$, we have that $\abs{y-tx}\geq\abs{y}/2$ for every $y\in B_R^c(0)$ and $t\in [0,1]$. Thus,
    \begin{equation*}
        \abs{\frac{1}{\abs{y-x}^{d+2s}}-\frac{1}{\abs{y}^{d+2s}}}=\abs{h_s(x)-h_s(0)}\leq \frac{(d+2s)\abs{x}2^{d+2s+1}}{\abs{y}^{d+2s+1}.}
    \end{equation*}
    Hence,
    \begin{equation}
    \label{eq:estimateII}
        II\leq \abs{x}(d+2s)\,s\int_{B^c_R(0)}\frac{|u(y)|}{\abs{y}^{d+2s+1}}\,dy\leq \frac{\abs{x}(d+2s)}{R}\,s\int_{B_R^c(0)}\frac{|u(y)|}{\abs{y}^{d+2s}}\,dy,
    \end{equation}
    which, by \eqref{Eq:alpha1R}, entails
    \begin{equation*}
        \lim_{R\to+\infty}\lim_{s\to 0^+}II=0.
    \end{equation*}
    This completes the proof of \eqref{claim:step2}. The statement of the lemma follows then by combining Steps 1 and 2.
\end{proof}

\begin{cor}
\label{lemma:lu-positive}
Let $u\in L^1_{loc}(\R^d)$, and assume that $\alpha(u^+),\alpha(u^-)\in \R$. Then, for every $p\in \N$, 
    \begin{equation}
        \label{luwelldefined-p}
        \begin{split}
             L_u^p&:=\lim_{R\to\infty}\liminf_{s\to 0^+}s\int_{B_R^c(x)}\frac{u(y)}{\abs{x-y}^{d+sp}}\,dy\\
             &=\lim_{R\to\infty}\limsup_{s\to 0^+}s\int_{B_R^c(x)}\frac{u(y)}{\abs{x-y}^{d+sp}}\,dy
        \end{split}
    \end{equation}
    is well defined and does not depend on $x\in\R^d$. Moreover, $L_u^p=\alpha_p(u)$.
\end{cor}
\begin{proof}
We omit the proof of this result, for it follows by repeating \emph{verbatim} the proof of Lemma \ref{lemma:lualphau} with different notational realizations.
\end{proof}

To prove Theorem \ref{Theo:MS pointwise:euclidean}, we will exploit a corollary of \cite[Proof of Theorem 2.2]{HanPinamontiXu2025}. To keep this paper self-contained, we recall its statement in the proposition below.

\begin{prop}[Theorem 2.2 in \cite{HanPinamontiXu2025}]
\label{prop:HanPinamonti}
    Let $(M, {d}, {m})$ be a metric measure space and let \mbox{$p\geq 1$}. Suppose there exists a sequence of non-negative, symmetric, measurable functions $\left(\rho_n\right)_{n \in \mathbb{N}}$ defined on $\{(x, y) \in M \times M: x \neq y\}$, such that:
    \begin{enumerate}[label=\Alph*)]
    \item \label{hA} there exists a constant $L \geq 0$ such that, for any $x \in M$,
    $$
    \lim _{R \rightarrow+\infty} \limsup_{n \rightarrow \infty} \int_{B_R^c(x)} \rho_n(x, y) {dm}(y)=\lim _{R \rightarrow+\infty} \liminf_{n \rightarrow \infty} \int_{B_R^c(x)} \rho_n(x, y) {dm}(y)=L
    $$
    where $B_R^c(x)=\{y \in M: {d}(x, y) \geq R\},$
    \item \label{hB} for any $u \in L^p(M)$ such that there exists $n_0 \in \mathbb{N}$ with
    $$
    \mathcal{E}_{n_0}(u):=\iint_{\{(x, y): x, y \in M, x \neq y\}}|u(x)-u(y)|^p \rho_{n_0}(x, y) dm(x) dm(y)<+\infty
    $$
    we have
    $$
    \lim _{n \rightarrow \infty} \iint_{\{(x, y): 0<{d}(x, y)<R\}}|u(x)-u(y)|^p \rho_n(x, y) dm(x) dm(y)=0, \quad \forall R>0,
    $$
    \item \label{hC} for any $R>0$ sufficiently large, there exists a constant $C=C(R)$ such that, for any $x \in M, n \in \mathbb{N}$,
    $$
    \int_{B_R^c(x)} \rho_n(x, y) {dm}(y) \leq C.
    $$
    \end{enumerate}
    Then, if $u\in L^p(M)$ is such that $\mathcal{E}_{n_0}(u)<+\infty$ for some $n_0\in \N$, there holds
    $$\lim_{n\to +\infty}\mathcal{E}_{n}(u)=2L\|u\|_{L^p(M)}^p.$$
\end{prop}
\begin{remark}
    \label{rem:CweakHanPinam}
    Note that in \cite[Proof of Theorem 2.2]{HanPinamontiXu2025} assumption \ref{hC} is only needed for the application of Fatou's lemma in the limsup estimate for the term $II_a(R,n)$. If the function $u$ has compact support, it is therefore enough to  show that for any $R>0$ there exists a constant $C=C(R)$ such that, for every $y\in \supp (u)$
    \begin{equation}
    %\label{eq:bdR}
        \limsup_{n\to \infty}\int_{B_R^c(y)}\rho_n(x,y)dm(x) \leq C(R).
    \end{equation}
    If this is true, then the thesis of Proposition \ref{prop:HanPinamonti} still holds even if Assumption \ref{hC} is not fulfilled.
\end{remark}

We now state three lemmata which will allow us to prove Theorem \ref{Theo:MS pointwise:euclidean}. For completeness, we state and prove them for every $p\geq 1$, since the proof doesn't depend on the specific value of $p$ considered.

The next lemma shows that the contribution of the interior-interior interactions in the Gagliardo seminorms vanishes as $s\to 0$.

\begin{lemma}
\label{lemma:Vanishes}
Let $u\in W^{s_0,p}(\O)$ for some $s_0\in(0,1)$. Then,
\begin{equation}
    \frac{s}{2}\semin{u}{W^{s,p}(\Omega)}^p\to0\quad\text{as } s\to0^+.
\end{equation}
\end{lemma}
\begin{proof}
The proof is a direct computation: for every $s<s_0$, we have
\begin{equation*}
    \begin{split}
        s\int_\O\int_\O\frac{\abs{u(x)-u(y)}^p}{\abs{x-y}^{d+sp}}dxdy&=s\int_\O\int_\O\frac{\abs{u(x)-u(y)}^p}{\abs{x-y}^{d+s_0p}}\abs{x-y}^{d+p(s_0-s)}dxdy\\
        &\leq \mathrm{diam}(\O)^{d+p(s_0-s)}s[u]_{W^{s_0,p}(\O)}^p\to0,
    \end{split}
\end{equation*}
which yields the thesis.
\end{proof}
\begin{remark}
\label{rk:seminormalimitati}
The same argument as in Lemma \ref{lemma:Vanishes} also yields that
\begin{equation*}
    \lim_{s\to 0^+}s\int_{\O^1}\int_{\O^2}\frac{\abs{u(x)-u(y)}^p}{\abs{x-y}^{d+sp}}dxdy=0,
\end{equation*}
for every pair $\O^1,\O^2$ of bounded domains in $\R^d$.
\end{remark}
We now show that, if $s_0$ is small enough, and $f$ and $g$ have finite mass at infinity, then the characteristic function of $\Omega$ has finite $s_0$-Gagliardo seminorm with respect to the measure
\begin{equation}\label{eq:def misura mista}
    \mu_{f,g}=f(x)\mathcal{L}^d_{\mid\O}+g(x)\mathcal{L}^d_{\mid\O^c}.
\end{equation}

\begin{lemma}
\label{lemma:product}
Let $p\geq1$, $f\in\elleom q$ for some $q>1$, $g\in\elle1_{\text{loc}}(\Rd)$, with $f,g\geq 0$, and assume that $\alpha_p(g)\in\R$. Then, there exists $s_0\in (0,1)$ such that
\begin{equation}\label{eq:chiOmega Sobolev}
    \int_{\R^d}\int_{\R^d}\frac{\abs{\chi_{\Omega}(x)-\chi_{\Omega}(y)}^p}{\abs{x-y}^{d+s_0p}}\,d\mu_{f,g}(y)\,d\mu_{f,g}(x)<+\infty,
\end{equation}
where $\mu_{f,g}$ is the measure defined in \eqref{eq:def misura mista}.
\end{lemma}
\begin{proof}
We rewrite the left-hand side of \eqref{eq:chiOmega Sobolev} as
\begin{equation*}
    2\int_\O\int_{\O^c}\frac{g(y)}{\abs{x-y}^{d+s_0p}}dy\,f(x)dx.
\end{equation*}
By the H\"older inequality, we obtain
\begin{equation*}
    \begin{split}
        \int_\O\int_{\O^c}\frac{g(y)}{\abs{x-y}^{d+s_0p}}dy\,f(x)dx\leq \left(\int_\O\left(\int_{\O^c}\frac{g(y)}{\abs{x-y}^{d+s_0p}}dy\right)^{q'}\,dx\right)^{\frac{1}{q'}}\cdot\left(\int_\O f(x)^qdx\right)^{\frac{1}{q}},
    \end{split}
\end{equation*}
where $q'<\infty$ is the H\"older conjugate exponent of $q$. We only need to show that the first integral is finite. For $x\in\O$ let $\delta_x=\mathrm{dist}(x,\partial \O)$. Then,
\begin{equation*}
    \begin{split}
        \int_{\O^c}\frac{g(y)}{\abs{x-y}^{d+s_0p}}\,dy&\leq\frac{1}{s}\int_{B_{\delta^c_x(x)}}\frac{s\,g(y)}{\abs{x-y}^{d+sp}}\abs{x-y}^{p(s-s_0)}\,dy\\
        &\leq \frac{1}{s\delta_x^{p(s_0-s)}}\int_{B_{\delta^c_x(x)}}\frac{s\,g(y)}{\abs{x-y}^{d+sp}}\,dy\\
        &\leq \frac{1}{s\delta_x^{p(s_0-s)}}(\alpha(g)+1)
    \end{split}
\end{equation*}
for $s<s_0$ sufficiently small. Integrating now with respect to $x$ in $\O$, we obtain
\begin{equation*}
    \begin{split}
        \int_\O\bigg(\int_{\O^c}\frac{g(y)}{\abs{x-y}^{d+s_0p}}&dy\bigg)^{q'}dx\leq \left(\frac{\alpha(g)+1}{s}\right)^{q'}\int_\O\frac{1}{\delta_x^{pq'(s_0-s)}}\,dx\\
        &=\left(\frac{\alpha(g)+1}{s}\right)^{q'}\int_0^{\mathrm{diam}(\O)}\frac{1}{r^{pq'(s_0-s)}}\Haus(\{x\in\O:\delta_x=r\})dr\\
        &\leq\left(\frac{\alpha(g)+1}{s}\right)^{q'}\int_0^{\mathrm{diam}(\O)}\frac{1}{r^{pq'(s_0-s)}}\Haus(\{x\in\O:\delta_x=r\})dr.
    \end{split}
\end{equation*}
By \cite[Section 3.2, Corollary 2]{Kraft2016}, there exists $C(\O)>0$ such that 
\[
\mathcal{H}^{d-1} \left\{ x\in \O:d(x,\partial \O)=r\right\}\leq C(\O)\qquad\text{ for every }r\in[0,\mathrm{diam(\O)}].
\]
Hence, we can estimate:
\begin{equation*}
      \int_\O\bigg(\int_{\O^c}\frac{g(y)}{\abs{x-y}^{d+s_0p}}dy\bigg)^{q'}dx\leq \left(\frac{\alpha(g)+1}{s}\right)^{q'}C(\O)\int_0^{\mathrm{diam}(\O)}\frac{1}{r^{pq'(s_0-s)}}dr<+\infty
\end{equation*}
if $s_0$ is sufficiently small. This concludes the proof.
\end{proof}

We conclude this subsection with an application of Proposition \ref{prop:HanPinamonti}.

\begin{lemma}
\label{lemma:alpha-int}
Let $f\in\elleom q$ for some $q>1$, $g\in L^1_{loc}(\Rd)$, with $f,g\geq 0$, and assume furthermore that $\alpha_p(g)$ exists and is finite. Then, there holds:
\begin{equation}
    \label{eq:lemmaalphaint}
     \lim_{s\to0^+} s\io\int_{\O^c}\frac{f(x)g(y)}{\abs{x-y}^{d+sp}}\,dy\,dx
     =
     \alpha_p(g)\int_{\O}f(x)dx.
\end{equation}
\end{lemma}
\begin{proof}
Let $(s_n)_n\subset(0,1)$ be any sequence such that $s_n\to0^+$. First, notice that the integral in the right-hand side of \eqref{eq:lemmaalphaint} can be rewritten, for $s=s_n$, as
\begin{equation*}
    \frac{1}{2}\int_{\R^d}\int_{\R^d}\abs{\chi_{\Omega}(x)-\chi_{\Omega}(y)}^p\rho_n(x,y)\,d\mu_{f,g}(y)d\mu_{f,g}(x),
\end{equation*}
where $\mu_{f,g}$ is is the measure defined in \eqref{eq:def misura mista} and 
$$\rho_n(x,y):=\frac{s_n}{|x-y|^{d+s_np}}\qquad\text{for every }n\in \mathbb{N}.$$
We start by showing that Assumptions \ref{hA} and \ref{hB} of Proposition \ref{prop:HanPinamonti} are fulfilled. Indeed, Assumption \ref{hA} holds due to Remark \ref{rk:alpha1}. As for Assumption \ref{hB}, by Lemma \ref{lemma:product}, we know that there exists $s_0\in(0,1)$, for which
\begin{equation*}
    \int_{\R^d}\int_{\R^d}\frac{\abs{\chi_{\Omega}(x)-\chi_{\Omega}(y)}^p}{\abs{x-y}^{d+s_0p}}\,g(y)dy\,f(x)dx<+\infty.
\end{equation*}
Then, whenever $s_n<s_0$ we have
\begin{equation*}
    \begin{split}
        &s_n\iint_{\{(x,y):0<|x-y|<R\}}\frac{\abs{\chi_{\Omega}(x)-\chi_{\Omega}(y)}^p}{\abs{x-y}^{d+s_np}}\,g(y)dy\,f(x)dx\\
        &\quad\leq s_nR^{(s_0-s_n)p} \int_{\R^d}\int_{\R^d}\frac{\abs{\chi_{\Omega}(x)-\chi_{\Omega}(y)}^p}{\abs{x-y}^{d+s_0p}}\,g(y)dy\,f(x)dx,
    \end{split}
\end{equation*}
which converges to $0$ as $s_n\to0^+$. 
As for Assumption \ref{hC}, this is not fulfilled in general without further requirements on $f$ and $g$. However, since the function $\chi_\Omega$ has compact support, recalling Remark \ref{rem:CweakHanPinam}, it is enough to show that for any $R>0$ there exists a constant $C=C(R)$ such that, for every $x\in \O$,
\begin{equation}
\label{eq:bdR}
    \limsup_{n\to \infty}s_n\int_{B_R^c(x)}\frac{g(y)}{\abs{x-y}^{d+s_np}}dy\leq C(R).
\end{equation}
Following the proof of \Cref{lemma:lualphau}, we obtain that, for $R>0$ big enough (depending only on $\O$), there exist $0<\delta_1<\delta _2$ (depending on $R$) such that $B_R(x)\Delta B_R(0)\subset B_{\delta_2}(0)\cap B_{\delta_1}^c(x)$ and 
the following estimate holds:
\begin{equation*}
    \begin{split}
        & \abs{  s_n\int_{B_R^c(x)}\frac{g(y)}{\abs{x-y}^{d+s_np}}dy-   s_n\int_{B_R^c(0)}\frac{g(y)}{\abs{y}^{d+s_np}}dy}\\
        &\hspace{2cm}\leq  \frac{s_n}{\delta_1^{d+s_np}}\int_{B_{\delta_2(0)}}g(y)\,dy\, +  \frac{\abs{x}(d+s_np)}{R}\,s_n\int_{B_R^c(0)}\frac{g(y)}{\abs{y}^{d+s_np}}\,dy.
  \end{split}
\end{equation*}    
Therefore,
\begin{equation*}
    \begin{split}
        &\limsup_{n\to \infty} \abs{  s_n\int_{B_R^c(x)}\frac{g(y)}{\abs{x-y}^{d+s_np}}dy-   s_n\int_{B_R^c(0)}\frac{g(y)}{\abs{y}^{d+s_np}}dy} \leq \frac{d\,\mathrm{diam(\O)}\alpha_p(g)}{R}.
  \end{split}
\end{equation*} 
By Lemma \ref{lemma:lualphau}, this entails
\begin{equation*}
    \begin{split}
        &\limsup_{n\to \infty}s_n\int_{B_R^c(x)}\frac{g(y)}{\abs{x-y}^{d+s_np}}dy\leq   \alpha_p(g)+ \frac{d\,\mathrm{diam(\O)}\alpha_p(g)}{R},
  \end{split}
\end{equation*} 
which, in turn, gives \eqref{eq:bdR}. We can now apply Proposition \ref{prop:HanPinamonti} to obtain
\begin{equation*}
    \frac{1}{2}\int_{\R^d}\int_{\R^d}\abs{\chi_{\Omega}(x)-\chi_{\Omega}(y)}^p\rho_n(x,y)\,d\mu_{f,g}(y)d\mu_{f,g}(x)=\alpha_p(g)\int_{\O}f(x)dx,
\end{equation*}
which concludes the proof.
\end{proof}

\subsection{Proof of Theorem \ref{Theo:MS pointwise:euclidean}}
We now prove Theorem \ref{Theo:MS pointwise:euclidean}. Note that we will apply the previous lemmata with $f,g\in\{1,u^+,u^-,u^2\}$, with $u\in H^{s_0}(\Qom)$ (and thus in $H^{s_0}(\O)$) for some $s_0\in(0,1)$, and such that $\alpha(u^+),\alpha(u^-),\alpha(u^2)$ exist and are finite. 
This, by fractional Sobolev embedding and since $u\in\elle2_\text{loc}(\Rd)$, ensures that all the assumptions of the lemmata above are satisfied.

\begin{proof}[Proof of Theorem \ref{Theo:MS pointwise:euclidean}]
    By expanding the square in the Gagliardo seminorm on $Q_\O$, and separating local and mixed contributions, we rewrite
    \begin{equation}
    \begin{split}
        \frac{s}{2}\semin{u}{H^s(Q_\Omega)}^2=&\frac{s}{2}\semin{u}{H^s(\Omega)}^2
        +s\io\int_{{\Omega^c}} \frac{\abs{u(x)}^2}{\abs{x-y}^{d+2s}}dydx\\
        &-2s\io\int_{{\Omega^c}} \frac{u(x)u(y)}{\abs{x-y}^{d+2s}}dydx
        +s\io\int_{{\Omega^c}} \frac{\abs{u(y)}^2}{\abs{x-y}^{d+2s}}dydx\\
        &:=I+II+III+IV     
    \end{split}
    \end{equation}
    We proceed by analyzing each of the above summands separately. First, by Lemma \ref{lemma:Vanishes}, $\lim_{s\to 0}I=0$.
    In view of Lemma \ref{lemma:alpha-int} (with $f=u^2$ and $g=1$), we find
    \begin{equation}
        \lim_{s\to0^+}   II=d\o_d\int_{\O}u(x)^2\,dx.
    \end{equation}
    To estimate III, recall $u^+(x)=u(x)\lor 0$ and $u^-(x)=-(u(x)\land 0)$. Clearly, 
    \begin{equation*}
        \begin{split}
            u(x)u(y)&=(u^+(x)-u^-(x))(u^+(y)-u^-(y))\\
            &=u^+(x)u^+(y)-u^+(x)u^-(y)-u^-(x)u^+(y)+u^-(x)u^-(y).
        \end{split}
    \end{equation*}
    Consider, for example, the term
    \begin{equation*}
         2s\io\int_{{\Omega^c}} \frac{u^+(x)u^-(y)}{\abs{x-y}^{d+2s}}dydx,
    \end{equation*}
    all the others being handled in a similar way. Applying Lemma \ref{lemma:alpha-int} (with $f=u^+$ and $g=u^-$), we find
    \begin{equation*}
 \lim_{s\to0^+}2s\io\int_{{\Omega^c}} \frac{u^+(x)u^-(y)}{\abs{x-y}^{d+2s}}dydx=2\alpha(u^-)\int_{\O}u^+(x)\,dx.
    \end{equation*}
    Thus,
    \begin{equation*}
        \begin{split}
            \lim_{s\to0^+}III & =-\alpha(u^+)\int_{\O}u^+(x)\,dx+\alpha(u^-)\int_{\O}u^+(x)\,dx\\
            &+\alpha(u^+)\int_{\O}u^-(x)\,dx-\alpha(u^-)\int_{\O}u^-(x)\,dx\\
            &=-\alpha(u)\int_{\O}u(x)\,dx.
        \end{split}
    \end{equation*}
    Eventually, a further application of Lemma \ref{lemma:alpha-int} (with $f=1$ and $g=u^2$), entails
    \begin{equation}
        \label{Eq:MSExtEuProof:IV}
        \lim_{s\to0^+}IV=\alpha(u^2)\L^d(\O),
    \end{equation}
    which concludes the proof.
\end{proof}

% \hrule

% \begin{example}(Asymptotics of the Euclidean fractional perimeter, see \cite{Quartet})
% Let $\Omega$ be a bounded domain in $\R^d$, let $E\in \R^d$ such that $\alpha_2(E)$ exists and define
% \begin{equation}
%     \operatorname{Per}_s(E;\Omega)=\frac12\semin{\chi_E}{H^s(Q_\Omega)}^2
% \end{equation}
%     The application of \Cref{Theo:MS pointwise:euclidean} to the function $u=\chi_E$ on the space $(\Omega,d_E,\mathcal{L}^d\,)$ leads to
%     \begin{equation}
%     \begin{split}
%             \lim_{s\to0^+}s\operatorname{Per}_s(E;\Omega)=&\alpha(\chi_E^2)\mathcal{L}^d(\Omega)+d\omega_d \io \chi_E^2\,d\mathcal{L}^d-2\alpha(E)\io\chi_E\,d\mathcal{L}^d\\
%             =& \alpha(E)\mathcal{L}^d(\Omega)+d\omega_d \mathcal{L}^d(E\cap \Omega)-2\alpha(E)\mathcal{L}^d(E\cap \Omega)\\
%             =& (d\omega_d-\alpha(E)) \mathcal{L}^d(E\cap \Omega)+\alpha(E)\mathcal{L}^d(\Omega\setminus E)
%     \end{split}
%     \end{equation}
%     where $\alpha(E)=\alpha(\chi_E)$. This is the result of \cite{Quartet}.
% \end{example}

\subsection{Interaction energies}
We now discuss a particular case of Theorem \ref{Theo:MS pointwise:euclidean} for functions which admit limits at infinity in every direction.
In particular, we will show that, assuming further hypotheses on the behavior at infinity of the functions involved, one can rewrite the formula in a more concise and natural fashion. This will also clarify the nature of the terms including $\alpha(u)$ (or $\alpha(u^2)$) in \eqref{Eq:intro:ExtMS:pw:eucl}.

We start by pointing out that, for functions $u=u(r,\theta)$ admitting limit as $r\to+\infty$ along every half-line starting from the origin, the value $\alpha(u)$ can be characterized as the integral of such limits along the angular component.

\begin{lemma}\label{lemma: integral at infinity}
    Let $u\colon [0,+\infty)\times \Sd\to\R$ be a bounded and measurable function such that for every $\theta\in\Sd$ the limit
    \[u_\infty(\theta)=\lim_{r\to+\infty} u(r,\theta)\]
    exists and, moreover, 
    \[\lim_{r\to+\infty}\norma{u(r,\cdot)-u_\infty}{\elle\infty(\Sd)}=0.\]
    Then
    \[
    \alpha_p(u)=\frac1p\int_\Sd u_\infty(\theta)d\Haus(\theta)
    \]
\end{lemma}
\begin{proof}
    Let $\e>0$ and $R>0$ be such that
    \[
    \norma{u(r,\cdot)-u_\infty}{\elle\infty(\Sd)}\leq \e\quad\forall r\geq R.
    \]
    Assume first that $u\geq 0$. Then, by Tonelli's theorem we can write 
    \[\int_{B_R^c(0)}\frac{u(y)}{\abs{y}^{d+sp}}dy= \int_R^{+\infty}\int_{\Sd}\frac{u(r,\theta)}{r^{d+sp}}d\Haus(\theta)\, dr.\]
    Therefore we have
    \begin{equation}
        \begin{split}
            \limsup_{s\to0^+}\int_{B_R^c(0)}\frac{u(y)}{\abs{y}^{d+sp}}dy
            &
            \leq \limsup_{s\to0^+}\int_R^{+\infty}\int_{\Sd}\frac{u_\infty(\theta)+\e}{r^{d+sp}}d\Haus(\theta)\, dr \\
            &\leq \Big(\int_\Sd u(\theta)d\Haus(\theta) +\e\Haus(\Sd)\Big)\limsup_{s\to0^+}\int_R^{+\infty}\frac1{r^{d+sp}}dr\\
            &\frac1p\Big(\int_\Sd u(\theta)d\Haus(\theta) +\e\Haus(\Sd)\Big),
        \end{split}
    \end{equation}
    and analogously: 
    \begin{equation}
        \begin{split}
            \liminf_{s\to0^+}\int_{B_R^c(0)}\frac{u(y)}{\abs{y}^{d+sp}}dy\geq \frac1p\Big(\int_\Sd u(\theta)d\Haus(\theta) -\e\Haus(\Sd)\Big).
        \end{split}
    \end{equation}
    Letting $\e\to0^+$, we get the thesis. In the general case it suffices to write $u=u^+-u^-$ and apply the previous step separately to $u^+$ and $u^-$. 
\end{proof}

\Cref{lemma: integral at infinity} corroborates the idea that the value $\alpha(u)$ only depends on the behaviour of $u$ at infinity, whenever such value is well-defined. This \say{integration at infinity} result is extended by the following corollary.
\begin{cor}\label{cor:composition alpha}
    Let $u$ satisfy the assumptions of \Cref{lemma: integral at infinity}. Then, for every continuous function $\psi:\R\to\R$, the value $\alpha_p(\psi(u))$ exists for every $p\geq 1$ and is given by
    \begin{equation}
        \alpha_p(\psi(u))=\frac1p 
        \int_\Sd \psi(u_\infty(\theta))d\Haus(\theta).
    \end{equation}
\end{cor}

We can now exploit this fact to show that in the case of functions that admit a continuous extension to the horizon the extended Maz'ya-Shaposhnikova formula \eqref{Eq:intro:ExtMS:pw:eucl} admits a simpler integral form, where the interaction between the part of $u$ inside $\Omega$ and its radial limit at infinity $u_\infty$ can be appreciated.

\begin{theo}\label{Theo:MS:integrale doppio}
     Let $\Omega\subset\R^d$ be a bounded domain and let $u\in L^2_{loc}(\R^d)$ be such that $u\in H^{s_0}(\Qom)$ for some ${s_0\in(0,1)}$. Moreover, assume that $u_{\mid\O^c}$ satisfies the assumptions of \Cref{lemma: integral at infinity}. Then:
    \begin{equation}\label{eq:MS euclidean p=2 raggruppata}
        \lim_{s\to0^+}\frac{s}{2}[u]_{H^s(Q_\O)}=\frac12\int_\O\int_{\Sd}|u(x)-u_\infty (\theta)|^2d\Haus(\theta)\,dx
    \end{equation}
\end{theo}
\begin{proof}
    From \Cref{cor:composition alpha} we know that all the assumptions of \Cref{Theo:MS pointwise:euclidean} are satisfied. It follows that
    \[
        \lim_{s\to 0^+}\frac{s}{2}[u]^2_{H^s(\Qom)}=\sum_{k=0}^{2}\binom{2}{k}(-1)^k\alpha(u^k)\int_\Omega u^{2-k}(x)\,dx.
    \]
    Note that, by \Cref{cor:composition alpha}, the right-hand side can be written as
    \begin{equation}\begin{split}
    &\frac12\sum_{k=0}^{2}\binom{2}{k}(-1)^k\int_\Sd u_\infty^k(\theta)d\Haus(\theta)\int_\Omega u^{2-k}(x)\,dx\\
    &=\frac12\sum_{k=0}^{2}\binom{2}{k}(-1)^k\io\int_\Sd u^{2-k}(x)u_\infty^k(\theta)d\Haus(\theta)\,dx\\
    &=\frac12\io\int_\Sd\sum_{k=0}^{2}\binom{2}{k}(-1)^k u^{2-k}(x)u_\infty^k(\theta)d\Haus(\theta)\,dx\\
    &=\frac12\io\int_\Sd\abs{u(x)-
    u_\infty(\theta)
    }^2d\Haus(\theta)\,dx
    \end{split}
    \end{equation}
\end{proof}
\begin{remark}
    Observe that for any function $u\in C^\infty_c(\R^d)$, \eqref{eq:MS euclidean p=2 raggruppata} reduces to
    \begin{equation}
        \lim_{s\to0^+}\frac{s}{2}[u]_{H^s(\Qom)}^2=\frac{1}2\Haus(\Sd)\io|u(x)|^2dx,
    \end{equation}
    which is the classical Maz'ya-Shaposhnikova formula for $p=2$.
\end{remark}

\begin{remark}
    \Cref{Theo:MS:integrale doppio} also clarifies why the limit of the fractional perimeters as $s\to0^+$ may fail to be a measure, a phenomenon first observed in \cite{Quartet}. Indeed, unless one restricts to functions that vanish at infinity, the limiting functional does not behave like a norm, but rather has the structure of an interaction energy.
\end{remark}

\subsection{Beyond $p=2$} 
As the reader might have noticed, the proof of Theorem \ref{Theo:MS pointwise:euclidean} exploits the fact that $\abs{u(x)-u(y)}^2$ can be expanded by the Newton formula for binomials. In this subsection, we show that an analogous result can be proved for any $p\in\N$ even and the limitations to extending it for $p\in\N$ odd. Then, we discuss possible extensions to the case of a general $p\in(1,+\infty)$.

First of all, we sketch the proof of a version of Theorem \ref{Theo:MS pointwise:euclidean} for every even $p\in\N$.
\begin{cor}
\label{cor:p-even}
 \label{Theo:MS pointwise:euclidean-p-even}
    Let $p\in \N$ be even, and let $u\in L^p_{loc}(\R^d)$ be such that $u\in W^{s_0,p}(\Qom)$ for some ${s_0\in(0,1)}$.
    Assume that  $\alpha_p(u_k^+), \alpha_p(u_k^-), \alpha(u^k)\in \R$, for every $k\in\{1,\dots,p\}$.
    Then:
    \begin{equation}
        \label{Eq:intro:ExtMS:pw:eucl-p-even}
        \lim_{s\to 0^+}\frac{s}{2}[u]^p_{W^{s,p}(\Qom)}=\sum_{k=0}^{p}\binom{p}{k}(-1)^k\alpha_p(u^k)\int_\Omega u^{p-k}(x)\,dx,
    \end{equation}
\end{cor}

\begin{proof}[Proof of Corollary \ref{cor:p-even}]
Assume, for simplicity, that $u$ only takes non-negative values.
We start by expanding $s\semin{u}{W^{s,p}(Q_\Omega)}^p$ as
    \begin{equation}
    \begin{split}
        \frac{s}{2}\semin{u}{W^{s,p}(Q_\Omega)}^p=&\frac{s}{2}\semin{u}{W^{s,p}(\Omega)}^p
        +s\int_{\Omega}\int_{\Omega^c} \sum_{k=0}^p \binom{p}{k}(-1)^k \frac{(u(y))^k (u(x))^{p-k}}{|x-y|^{d+sp}}dydx. 
    \end{split}
    \end{equation}
First, note that Lemma \ref{lemma:Vanishes} yields
\[\frac{s}{2}[u]^p_{W^{s,p}(\O)}\to 0\]
as $s\to 0$. Then, by Lemma \ref{lemma:alpha-int},
\[
\lim_{s\to0}s\int_{\Omega}\int_{\Omega^c}\frac{(u(x))^k (u(y))^{p-k}}{|x-y|^{d+sp}}dydx=\alpha_p(u^k)\io u^{p-k}(x)dx,
\]
which concludes the proof.
\end{proof}

If $p\in\N$ is odd, we cannot recover a closed formula such as \eqref{Eq:intro:ExtMS:pw:eucl-p-even}. Indeed, for any $u\in\C^\infty_c(\Omega)$, the classical Maz'ya-Shaposhnikova formula states that
\[
\lim_{s\to0^+}\frac{s}{2}[u]_{W^{s,p}(Q_\O)}^p=\frac{d\omega_d}{p}\io \abs{u(x)}^pdx,
\]
while the right-hand side of \eqref{Eq:intro:ExtMS:pw:eucl-p-even} would just reduce to $\io u(x)^pdx$. The following example shows some technical difficulties which would arise if one tried to follow this approach.
\begin{example}
    \label{poddcounterexample}
    Choose any $p\in\N$. Take $d=1$ and $\O=(-1,1)$. Consider:
    \begin{equation*}
        u(x):=\\
        \begin{cases}
            x\qquad\mathrm{if}-1<x<1\\
            0\qquad\mathrm{otherwise},
        \end{cases}
    \end{equation*}
    and then set $v=u+1$.  Note that, for every $x,y\in\Rd$,
    \[
    u(x)-u(y)=v(x)-v(y).
    \]
    Since $u_{\mid\O^c}=0$, it follows that
    \[
    \semin{v}{W^{s,p}(\Qom)}^p=\semin{u}{W^{s,p}(\Qom)}^p=\semin{u}{W^{s,p}(\Rd)}^p.
    \]
    Thus, by the Maz'ya-Shaposhnikova formula \eqref{eq:intro:classicalMS}, 
    \[
    \lim_{s\to0^+} \frac{s}{2}[v]_{W^{s,p}(\Qom)}^p=
    \frac12\lim_{s\to0^+} s[u]_{W^{s,p}(\Rd)}^p=\frac12\frac{2}{p}\int_{-1}^1\abs{x}^pdx=\frac{2}{p(p+1)}.
    \]
    On the other hand, we have that:
    \begin{align*}
            \sum_{k=0}^p\binom{p}{k}(-1)^k\alpha_p(v^k)\int_\O v^{p-k}dx&=\alpha_p(1)\int_\O\sum_{k=0}^p\binom{p}{k}(-1)^k v^{p-k}dx\\
            &=\frac{2}{p}\int_\O(v-1)^pdx=\frac{2}{p}\int_{-1}^1x^p\,dx,
    \end{align*}
    which is equal to zero if $p$ is odd. 
\end{example}

On the one hand, the example above suggests that an explicit formula as \eqref{Eq:intro:ExtMS:pw:eucl-p-even} might be beyond reach for general $p\neq 2$. 
On the other hand, it is natural to conjecture that the result in Theorem \ref{Theo:MS:integrale doppio} might extend to the case $p\neq2$. The extent to which this result can be generalized and the underlying assumptions can be relaxed is investigated in \cite{DaFaPi2}.

\subsection{A critical case}

It was shown in \cite{Quartet} that, for measurable sets $E\subset\R^d$ such that
\begin{equation}\label{eq:halfmeasure}
    \L^d(E\cap\O)=\L^d(E^c\cap \O),
\end{equation}
the limit of $s\operatorname{Per}_s(E;\O)$ exists independently of the existence of $\alpha_1(E)$ and
\[
\lim_{s\to0^+}s\operatorname{Per}_s(E;\O)=d\o_d\L^d(E\cap\O).
\]
In this subsection, we prove that
a similar result holds for the $H^s(\Qom)$ seminorm. Although its formulation, having to encompass many more cases, is slightly more involved,  it reduces to \eqref{eq:halfmeasure} when restricted to characteristic functions of sets.

\begin{prop}\label{prop:casocritico}
    Let $s_0>0$ and let $u\in H^{s_0}(\Qom)$. For every $y\in\O^c$, define
    \[
    f(y):=\io\abs{u(x)-u(y)}^2\,dx.
    \]
    Then, if $\alpha(f)\in\R$, it holds
    \begin{equation}
        \lim_{s\to0^+}\frac{s}{2}\semin{u}{H^s(\Qom)}^2=\alpha(f).
    \end{equation}
\end{prop}
Note that condition $\alpha(f)\in\R$ holds trivially if $f$ is constant. This is the case, for example, if $u=\chi_E$ for a set $E$ satisfying \eqref{eq:halfmeasure}. Indeed,
\begin{equation}
    \begin{split}
        \io\abs{\chi_E(x)-\chi_E(y)}^2\,dx&=m(E\cap\O)-2m(E\cap\O)\chi_E(y)+m(\O)\chi_E(y)\\
        &=m(E\cap\O)\chi_{E^c}(y)+m(E^c\cap\O)\chi_E(y)\\
        &=m(E\cap\O).
    \end{split}
\end{equation}
This shows, in particular, that the critical case in \cite{Quartet} can be recovered as a corollary of Proposition \ref{prop:casocritico}.
% The proof of Proposition \eqref{prop:casocritico} requires the following Lemma, which can be proved by following the proof of Lemma \ref{lemma:lualphau}.
% \begin{lemma}\label{lemma:limsup_alpha}
%     Let $u\in\elle1_{\text{loc}}(\Rd)$ and assume that
%     \begin{equation}
%         \limsup_{s\to0^+} s\int_{B_1^c} \frac{u(y)}{\abs{y}^{d+2s}}dy<+\infty.
%     \end{equation}
%     Then, for every $R>0$ and $x\in \Rd$
%     \begin{equation}
%         \lim_{s\to0^+} \abs{s\int_{B_1^c(0)} \frac{u(y)}{\abs{y}^{d+2s}}dy-s\int_{B_R^c(x)}\frac{u(y)}{\abs{x-y}^{d+2s}}dy}=0.
%     \end{equation}
% \end{lemma}

\begin{proof}%[Proof of Proposition \ref{prop:casocritico}]

Let $R>0$ be such that $\O\sub B_R(0)$. Arguing as in the proof of \eqref{eq:estimateII} in Lemma \ref{lemma:lualphau}, we obtain that, for every $x\in \O$ and $y\in B_R^c(0)$,
\begin{align*}
        &\left|s\int_{B_R^c(0)}\frac{|u(x)-u(y)|^2}{|x-y|^{d+2s}}dy\,-s\int_{B_R^c(0)}\frac{|u(x)-u(y)|^2}{|y|^{d+2s}}dy\right|\\&\hspace{8mm}\leq\frac{(d+2s)\mathrm{diam}(\O)}{R}s\int_{B_R^c(0)}\frac{|u(x)-u(y)|^2}{|y|^{d+2s}}dy.
\end{align*}
Hence, integrating with respect to $x\in\O$:
\begin{align*}
    &\left|s\int_\O\int_{B_R^c(0)}\frac{|u(x)-u(y)|^2}{|x-y|^{d+2s}}dy\,dx\,-s\int_\O\int_{B_R^c(0)}\frac{|u(x)-u(y)|^2}{|y|^{d+2s}}dy\,dx\right|\\&\qquad\qquad\leq\frac{(d+2s)\mathrm{diam}(\O)}{R}s\int_\O\int_{B_R^c(0)}\frac{|u(x)-u(y)|^2}{|y|^{d+2s}}dy.\\
\end{align*}
By taking the limsup as $s\to0^+$ on both sides, recalling the assumption $\alpha(f)\in\R$, we obtain:
\begin{align*}
     &\limsup_{s\to0^+}\left|s\int_\O\int_{B_R^c(0)}\frac{|u(x)-u(y)|^2}{|x-y|^{d+2s}}dy\,dx\,-s\int_\O\int_{B_R^c(0)}\frac{|u(x)-u(y)|^2}{|y|^{d+2s}}dy\,dx\right|\\&\hspace{4cm}\leq\frac{d\,\mathrm{diam}(\O)}{R}\alpha(f),\\
\end{align*}
which goes to zero as $R\to\infty$. We can therefore compute:
    \begin{align*}
         \lim_{s\to0^+}\frac{s}{2}\semin{u}{H^s(\Qom)}^2&=\lim_{s\to0^+}s\int_{\O}\int_{\O^c} \frac{|u(x)-u(y)|^2}{|x-y|^{d+2s}}dxdy \\
         &=\lim_{R\to +\infty}\lim_{s\to0^+}s\int_{\O}\int_{B_R^c(0)} \frac{|u(x)-u(y)|^2}{|x-y|^{d+2s}}dxdy \\
         &=\lim_{R\to +\infty}\lim_{s\to0^+}s\int_{\O}\int_{B_R^c(0)} \frac{|u(x)-u(y)|^2}{|y|^{d+2s}}dxdy\\
         &=\lim_{R\to +\infty}\lim_{s\to0^+}s\int_{B_R^c(0)} \frac{f(y)}{|y|^{d+2s}}dxdy\\
         &=\alpha(f),
    \end{align*}
    where in the first equality we have used Remark \ref{rk:seminormalimitati}.
\end{proof}

\section{Gamma convergence}\label{Sec:Gammaconv}
In this section we show that the limiting formula identified in Theorem \ref{Theo:MS pointwise:euclidean} also captures the asymptotic behavior of the fractional Gagliardo seminorms in the sense of $\Gamma$-convergence. In particular, we prove that, provided that suitable solid boundary data are fixed outside the bounded domain $\O$, $\Gamma$-convergence with respect to the weak-$L^2(\O)$ topology can be proved with little effort. In what follows, without loss of generality, we will always assume that $\Omega$ contains the origin.

For this purpose, let us start by defining the space in which we will set our analysis. 
Let
\begin{equation}
    \mathcal{A}(\O):=\left\{v\in L^\infty (\O^c):\exists\;\alpha(v^+),\, \alpha(v^-),\,\alpha(v^2)\right\},
\end{equation}
be the set of all admissible boundary data. For any $v\in\mathcal{A}(\O)$, we consider the space 
\begin{equation}
    L^2_{loc}(Q_\O,v)=\left\{u\in L^2_{loc}(\R^n):u=v\textit{ on }\O^c\right\}.
\end{equation}
Finally, for every $s\in(0,1)$ and every $u\in L^2_{loc}(\R^n)$ we denote:
\begin{equation*}
    F_s(u,\O):=\begin{cases}\frac{s}{2}[u]_{H^s(Q_\O)}^2&\text{if }u\in L^2_{loc}(Q_\O,v),\\
    +\infty&\text{otherwise}.
    \end{cases}
\end{equation*}
First, we prove equi-coercivity for the functionals $F_s(\cdot,\O)$ with respect to the weak-$L^2(\O)$ topology. This result is a consequence of the fractional Hardy inequality proved in \cite[Formula (9)]{MazyaShaposhnikova2002}, which we recall here in the case $p=2$ (see also \cite[Theorem 1.1]{CriDLKNP23}). For every $s\in(0,1)$, let $H^s_0(\R^d)$ be the completion of $C^\infty_c(\R^d)$ with respect to $[\cdot]_{H^s(\R^d)}$ (note that the Gagliardo seminorm is, in fact, a norm in such space). By classical approximation results (see \cite[Theorem 6.78]{LeoniSobolevSpaces}), we have that, if $s\leq 1$, $H^s_0(\R^d)=H^s(\R^d)$. The fractional Hardy inequality in this setting reads as follows.
\begin{prop}[Formula (9) in \cite{MazyaShaposhnikova2002}]
   \label{Theo:HardyMS02}
   Let $\delta\in(0,1)$ and take $s\in(0,\frac{\delta^2}{8})$. Then, for every $u\in H^s_0(\R^d)$,
   \begin{equation}
        \label{Eq:HardyMS}
        \frac{d\omega_d(1-\delta)^2}{2^{2s+1}}\int_{\R^d}\frac{|u(x)|^2}{|x|^{2s}}dx\leq \frac{s}{2}[u]_{H^s(\R^d)}^2.
   \end{equation}
\end{prop}
\begin{remark}
    Let $\O\subseteq\R^d$ be a bounded domain with Lipschitz boundary. For any $u\in H^s(\O)$ let:
     \begin{equation*}
        \tilde{u}(x):=
        \begin{cases}
            u(x)& \text{if } x\in\O\\
            0 & \text{if } x\in\O^c.
        \end{cases}
    \end{equation*}
    For $s<\frac12$, we can apply \ref{Eq:HardyMS} to $\tilde{u}$ to obtain:
    \begin{equation}
        \label{Eq:HardyMSOmega}
         \frac{d\omega_d(1-\delta)^2}{2^{2s+1}}\int_{\O}\frac{|u(x)|^2}{|x|^{2s}}dx\leq \frac{s}{2}[\tilde u]_{H^s(\R^d)}^2.
    \end{equation}
\end{remark}
Our compactness result reads then as follows.
\begin{prop}
\label{prop:compactness}
    Let $v\in\mathcal{A}(\O)$ and let $(u_s)_s\subseteq L^2_{loc}(Q_\O,v)$, $s\in(0,1)$, be such that
    \begin{equation*}
        \sup_{s\in(0,1)}F_s(u)<\infty.
    \end{equation*}
    Then, there exists a function $u\in L^2(\O)$ such that, up to subsequences, $u_s\rightharpoonup u$ weakly in $L^2(\O)$.
\end{prop}
\begin{proof}
    For any $s\in(0,1)$, let $\tilde{u}_s$ denote the zero extension of $u_s$ outside $\O$. We estimate:
    \begin{align*}
        \frac{s}{2}[\tilde{u}_s]_{H^s(\R^d)}^2&=\frac{s}{2}\int_\O\int_\O\frac{|u_s(x)-u_s(y)|^2}{|x-y|^{d+2s}}dydx+s\int_\O\int_{\O^c}\frac{|u_s(x)|^2}{|x-y|^{d+2s}}dydx\\
        &\leq s[u_s]_{H^s(Q_\O)}^2+2s\int_\O\int_{\O^c}\frac{|v(y)|^2}{|x-y|^{d+2s}}dydx.
    \end{align*}
    By the definition of $\mathcal{A}(\O)$, the same argument as in \eqref{Eq:MSExtEuProof:IV} implies that there exists $\tilde{s}\in (0,1)$ such that, for every $s<\tilde{s}$
    \begin{equation*}
        s\int_\O\int_{\O^c}\frac{|v(y)|^2}{|x-y|^{d+2s}}dydx\leq \alpha(v^2)\L^d(\O)+1:=c(\O,v),
    \end{equation*}
    and, therefore,
    \begin{equation}
        \label{Eq:GammProofEst1}
        \sup_{s<\tilde{s}}\frac{s}{2}[\tilde{u}_s]_{H^s(\R^d)}^2\leq \sup_{s<\tilde{s}} s [u_s]_{H^s(Q_\O)}^2 + c(\O,v)
        <\infty .   
    \end{equation}
    By choosing $\tilde{s}\in(0,1/32)$, and applying \eqref{Eq:HardyMSOmega} with $\delta=1/2$, we get, for $s<\tilde s$, 
    \begin{equation*}
         \frac{d\omega_d}{2^{2s+3}\mathrm{diam}(\O)^{2s}}\int_{\O}|u_s(x)|^2dx \leq  \frac{s}{2}[\tilde{u}_s]_{H^s(\R^d)}^2 ,
    \end{equation*}
    which, together with \eqref{Eq:GammProofEst1}, gives
    \begin{equation}
         \sup_{s<\tilde{s}}\|u_s\|^2_{L^2(\O)} <\infty,
    \end{equation}
    thus ending the proof.
\end{proof}
We now state and prove the main result of this section.
\begin{theo}
    \label{Theo:MSGamma}
    Let $\O\subseteq \R^d$ be a bounded domain with Lipschitz boundary and let $v\in\mathcal{A}(\O)$. Then, as $s\to 0^+$, the family of functionals $F_s(u,\O)$ Gamma-converges with respect to the weak-$L^2(\O)$ topology to the functional:
    \begin{equation}
        F_0(u,\O):=
        \begin{cases}
            \sum_{k=0}^2 (-1)^k\binom{2}{k}\alpha(v^k)\io u^{2-k}(x)dx &\text{if }u\in L^2(\O)\\
            +\infty &\text{otherwise.}
        \end{cases}
    \end{equation}
\end{theo}
\begin{proof}
    By Theorem \ref{Theo:MS pointwise:euclidean}, we only need to prove that whenever $(u_s)_s\subseteq L^2_{loc}(Q_\O,v)$ converges weakly-$L^2(\O)$ to some $u\in L^2(\O)$, then
    \begin{equation}
        \liminf_{s\to0^+}F_s(u_s,\O)\geq \sum_{k=0}^2 (-1)^k\binom{2}{k}\alpha(v^k)\io u^{2-k}(x)dx.
    \end{equation}
    Let us start by decomposing the energy in the same way as in the proof of Theorem \ref{Theo:MS pointwise:euclidean}:
    \begin{align*}
        F_s(u_s,\O)=&\frac{s}{2}\semin{u_s}{H^s(\Omega)}^2
        +s\io\int_{{\Omega^c}} \frac{\abs{u_s(x)}^2}{\abs{x-y}^{d+2s}}dydx\\
        &-2s\io u_s(x)\int_{{\Omega^c}} \frac{v(y)}{\abs{x-y}^{d+2s}}dydx
        +s\io\int_{{\Omega^c}} \frac{\abs{v(y)}^2}{\abs{x-y}^{d+2s}}dydx\\
        &:=I+II+III+IV.     
    \end{align*}
    The same proof as for \eqref{Eq:MSExtEuProof:IV} yields that $IV$ converges to $\alpha(v^2) \L^d(\O)$ as $s\to0^+$. Now, notice that:
    \begin{equation*}
        I+II=\frac{s}{2}[\tilde{u}_s]_{H^s(\R^d)}^2.
    \end{equation*}
    We can therefore exploit the argument used in the proof of  \cite[Theorem 1.2]{CriDLKNP23}. Indeed by Proposition \ref{Theo:HardyMS02} we have
    \begin{align*}
        \frac{s}{2}[\tilde u_s]_{H^s(\R^d)}^2&\geq \frac{d\omega_d(1-\delta)^2}{2^{2s+1}}\int_{\O}\frac{|u_s(x)|^2}{|x|^{2s}}dx\\
        &\geq \frac{d\omega_d(1-\delta)^2}{2^{2s+1}\mathrm{diam}(\O)^{2s}}\int_{\O}|u_s(x)|^2dx,
    \end{align*}
    which yields:
    \begin{equation*}
        \liminf_{s\to0^+}\frac{s}{2}[\tilde u]_{H^s(\R^d)}^2\geq \frac{d \omega_d}{2}(1-\delta)^2\int_{\O}|u(x)|^2dx.
    \end{equation*}
    Letting $\delta\to 0^+$ we finally obtain:
    \begin{equation}
         \liminf_{s\to0^+}\frac{s}{2}[\tilde u]_{H^s(\R^d)}^2\geq \frac{d \omega_d}{2}\int_{\O}|u(x)|^2dx.
    \end{equation}
    To deal with $III$, we start noticing that from \Cref{lemma:lualphau} it follows that:
    \begin{equation}
        \label{Eq:Gammaproof:alphavpw}
        \lim_{s\to0^+}s\int_{\O^c}\frac{v(y)}{|x-y|^{d+2s}}dy=\alpha(v).
    \end{equation}
    
    For $x\in\O$, we set
    \begin{equation*}
        g_s(x):=s\int_{\O^c}\frac{v(y)}{|x-y|^{d+2s}}dy,
    \end{equation*}
    and use the notation $\delta_x:=d(x,\partial \Omega)$. Let $\tilde{v}$ denote the extension by zero of $v$ to $\O$. For $s<1/8$, we estimate
    \begin{align*}
         g_s(x)= &\, s\int_{B_{\delta(x)}^c(x)}\frac{\tilde{v}(y)}{|x-y|^{d+2s}}dydx\leq \|v\|_{L^\infty(\Omega^c)} \,d\omega_d\,s\int_{\delta(x)}^\infty\frac{1}{r^{1+2s}}dr\\
        &=\|v\|_{L^\infty(\Omega^c)}\frac{d\omega_d}{2}\frac{1}{\delta(x)^{2s}}\\
        &\leq \|v\|_{L^\infty(\Omega^c)}\frac{d\omega_d}{2}\left(\chi_{\{\delta\geq 1\}}(x)+\frac{\chi_{\{\delta<1\}}(x)}{\delta(x)^{1/4}}\right):=G(x).
    \end{align*}
    We have that:
    \begin{align*}
        \int_\O G(x)^2dx&=\left(\|v\|_{L^\infty(\Omega^c)}\frac{d\omega_d}{2}\right)^2\left(\L^d(\O)+\int_0^{\mathrm{diam}(\O)}\frac{1}{r^{1/2}}\mathcal{H}^{d-1}(\{\delta(x)=r\})dr\right)\\
        &\leq C(\O)\left(\|v\|_{L^\infty(\Omega^c)}\frac{d\omega_d}{2}\right)^2\int_0^{\mathrm{diam}(\O)}\frac{1}{r^{1/2}}dr<\infty,
    \end{align*}
    for a suitable constant $C(\O)$ only dependent on the set $\O$.
    Thus, by the Dominated Convergence Theorem we obtain
    \begin{equation*}
        \lim_{s\to0^+}\|g_s-\alpha(v)\|_{L^2(\O)}=0,
    \end{equation*}
    which together with the weak-$L^2(\O)$ convergence of $(u_s)_s$ gives:
    \begin{equation}
        \lim_{s\to0^+}s\int_\O u_s(x)\int_{\O^c}\frac{v(y)}{|x-y|^{d+2s}}dydx=\alpha(v)\int_\O u(x)dx.
    \end{equation}
    Then, the thesis follows by combining the convergence of all terms I,II,III, and IV.
\end{proof}
\begin{remark}
    We point out that the same result holds under slightly weaker assumptions on the boundary data. In fact, a careful inspection of the proof, as well as the continuity of the mass at infinity with respect to the strong $L^\infty$-convergence of its argument, shows that an analogous Gamma-convergence analysis can be established for the sequence of functionals
    \begin{equation*}
    \widetilde{F_s}(u,\O):=\begin{cases}\frac{s}{2}[u]_{H^s(Q_\O)}^2&\text{if }u\in L^2_{loc}(Q_\O,v_s),\\
    +\infty&\text{otherwise},
    \end{cases}
\end{equation*}
provided that there exists $v\in L^{\infty}(\O^c)$ such that the data satisfy $v_s\to v$  strongly in $L^{\infty}(\O^c)$ as $s\to 0^+$.\end{remark}

\begin{remark}
    The fact that the topology for the $\Gamma$-limit is the weak $\elle2$ topology (which appears to be the optimal one even in the case of vanishing boundary data \cite{CriDLKNP23}) highlights that an extension in the sense of $\Gamma$-convergence of \eqref{eq:intro:quartet} (see \cite[Theorem 2.5]{Quartet}) would not lead to a functional defined on sets, but one defined on measurable functions attaining values in $[0,1]$. 
    In this sense, \Cref{Theo:MSGamma} provides the lower semicontinuous envelope (with respect to the weak-$\elleom2$ topology) of the functional
    \[
    \mu(u)=\begin{cases}
        (d\omega_d-\alpha_1(E))\mathcal{L}^d(\O\cap E)+\alpha_1(E)\mathcal{L}^d(\O\setminus E )& \text{if }u=\chi_E\text{ and }\exists \alpha_1(E)\\
        +\infty&\text{ otherwise}.
    \end{cases}
    \]
    Indeed, take a measurable set $F\subseteq \R^d$ such that $\chi_F\in\mathcal{A}(\O)$, \textit{i.e.} $\alpha_1(F)$ exists. 
    Then the lower semicontinuous envelope of $\mu$ with respect to the weak-$L^2(\O)$ topology on the space $L^2_{loc}(Q_\O,\chi_F)$ is given by: 
    \[
    \mu^{**}(u)=\begin{cases}
        2F_0(u,\O)\qquad&\text{if }u:\Omega\to[0,1]\\
        +\infty\qquad&\text{otherwise}.
    \end{cases}
    \]
    where the multiplication by $2$ comes from the identity $2\alpha_2(u)=\alpha_1(u)$ for every $u$ with finite mass at infinity.
\end{remark}

\section{The metric setting}\label{Sec:puntuale mms}
In this section we extend \Cref{Theo:MS pointwise:euclidean} to the metric measure framework. A metric measure space is a triple $(M,d,m)$ such that $(M,d)$ is a Polish space and $m$ is a positive Borel measure which is finite on bounded sets.
The validity of a Maz’ya–Shaposhnikova formula in the metric measure setting was first established in \cite{Han-MS&BBM}.

In this context, the analogue of the Gagliardo seminorm is defined by means of a measurable function $\rho:M\times M\to [0,+\infty)$. We set
\begin{equation}
    \semin{u}{M,\rho}^2=\int_M\int_M \abs{u(x)-u(y)}^2\rho(x,y)dm(x)dm(y)
\end{equation}
and, for any open subset $\Omega\subset M$, the Gagliardo seminorm with interior-interior interactions is defined as
\begin{equation}
    \semin{u}{\Qom,\rho}^2=\iint_\Qom \abs{u(x)-u(y)}^2\rho(x,y)dm(x)dm(y).
\end{equation}

Since our goal is, in a loose sense, to understand the behaviour of the $H^s$ Gagliardo seminorms as $s\to0^+$, it is natural to consider a family of convolution kernels $\{\rho_n\}_n$ that "spread out" as $n$ tends to infinity. We refer the reader to \cite{DaDfGioPi-Necess&Suff} for a detailed discussion of the structural assumptions required for a Maz’ya–Shaposhnikova formula to hold.

The assumptions adopted here combine those in \cite{HanPinamontiXu2025} with additional ones which ensure a control on the interaction between the energies in the interior and the exterior of the reference domain. Roughly speaking, these assumptions play the role of the results obtained in \Cref{Sec:puntuale Rd} by means of the fractional Sobolev embedding and of the coarea formula. Once again, we restrict our attention to the case $p=2$.

Let $\Omega$ be an open subset of $M$. For any pair of nonnegative, measurable functions $v$ and $w$, define the measure $\mu_{v,w}$ on $M$ as
\begin{equation}
    \mu_{v,w}(E)=\int_{\Omega\cap E} v\,dm+\int_{\Omega^c\cap E} w\,dm,
\end{equation}
which can also be expressed as 
\[
    \mu_{v,w}=v\chi_\Omega\,dm+w\chi_{\Omega^c}\,dm.
\]
Then, for any $v,w\in\elle1_{loc}(M)$, we define $\mu_{v,w}$ as the Radon functional given by
\[
\mu_{v,w}=\mu_{v^+,w^+}+\mu_{v^-,w^-}-\mu_{v^+,w^-}-\mu_{v^-,w^+}.
\]
Moreover, given a measurable function $v$, we define (whenever possible)
    \begin{equation}
        \alpha(v)(x)=\lim_{n\to+\infty}\int_{B_R^c(x)}{v(y)\rho_n(x,y)}dm(y).
    \end{equation}

We can now state the main assumptions we make to extend \Cref{Theo:MS pointwise:euclidean} to the metric measure setting. In what follows, $u$ will be a real-valued measurable function defined on $M$.

The first requirement is the content of \Cref{lemma:lualphau}.
\begin{hyp}\label{Ass:alphaexist}
    The limits $\alpha(\chi_{\O})(x)$, $\alpha(\chi_{\O^c})(x)$, $\alpha(u^+\chi_{\O^c})(x)$, $\alpha(u^-\chi_{\O^c})(x)$, $\alpha(u^2\chi_{\O^c})(x)$ exist and are independent of $x$ and $R$. 
\end{hyp}

The second assumption is the (fractional) Sobolev-regularity of the characteristic function of $\Omega$, which was established in the Euclidean setting through \Cref{lemma:product}

\begin{hyp}\label{Ass:chiOmega Sobolev}
    There exists $n_0\in\N$ such that the characteristic function $\chi_\Omega$ satisfies
    \begin{equation}
        \int_M\int_M {\abs{\chi_\Omega(x)-\chi_\Omega(y)}}^2\rho_{n_0}(x,y) \,d\mu_{v,w}(x,y) <+\infty
    \end{equation}
    for every choice of $(v,w)\in \{u^+,u^-\}^2\cup \{(1,u^2), (u^2,1)\}$.
\end{hyp}

The last two assumptions are needed for the validity of the metric Maz'ya-Shaposhnikova formula proved in \cite{HanPinamontiXu2025}.
\begin{hyp}\label{Ass:vanishes_on_bounded} 
For any $u \in L^2(M)$ such that there exists $n_0 \in \mathbb{N}$ with
$$
\mathcal{E}_{n_0}(u):=\iint_{\{(x, y): x, y \in M, x \neq y\}}|u(x)-u(y)|^2 \rho_{n_0}(x, y) dm(x) dm(y)<+\infty
$$
we have
\begin{equation}
    \label{Eq:ass_mass_at_infinity}
    \lim _{n \rightarrow \infty} \iint_{\{(x, y): 0<{d}(x, y)<R\}}
|u(x)-u(y)|^2 \rho_n(x, y) dm(x) dm(y)=0, \quad \forall R>0.
\end{equation}

\end{hyp}

\begin{hyp}\label{Ass:finite mass at infinity}
For any $R>0$ sufficiently large, there exists a constant $C=C(R)$ such that, for every $x \in M$ and $n \in \mathbb{N}$,
$$
\int_{B_R^c(x)} \rho_n(x, y) {dm}(y) \leq C.
$$
\end{hyp}

We now state and prove the extension of Theorem \ref{Theo:MS pointwise:euclidean} to the metric measure setting
\begin{theo}\label{Theo:MS pointwise:mms}
    Let $(M,d,m)$ be a metric measure space, $\Omega$ be an open subset of $M$ and let $\{\rho_n\}_n$ be a sequence of of non-negative, symmetric, measurable convolution kernels and $u$ such that $\semin{u}{\Qom,\rho_{n_0}}<+\infty$.
    Assume that 
    \begin{itemize}
        \item The kernels $\{\rho_n\}_n$ restricted to $\Omega\times \Omega$ satisfy Assumptions \ref{Ass:vanishes_on_bounded} and \ref{Ass:finite mass at infinity};
        \item \Cref{Ass:alphaexist} and \Cref{Ass:chiOmega Sobolev} hold;
        \item For every choice of $(v,w)\in \{u^+,u^-\}^2\cup \{(1,u^2), (u^2,1)\}$, the kernels $\{\rho_n\}_n$ satisfy Assumptions \ref{Ass:vanishes_on_bounded} and \ref{Ass:finite mass at infinity} on $(M,d,\mu_{v,w})$.
    \end{itemize}
    Then,
    \begin{equation}
    \begin{split}
        \lim_{n\to+\infty}\semin{u}{Q_\Omega,\rho_n}^2
        =
        &\alpha(1)\io u^2\,dm
        -2\alpha(u\chi_{\Omega^c})\io u\,dm
        +\alpha(u^2\chi_{\Omega^c})m(\Omega)
    \end{split}
    \end{equation}
\end{theo}

\begin{proof}
    We start by expanding (with a small abuse of notation in the use of $\mu_{u,u}$) the seminorm $\semin{u}{Q_\Omega,\rho_n}^2$ as
    \begin{equation}
    \begin{split}
        \semin{u}{Q_\Omega,\rho_n}^2&
        =
        \semin{u}{\Omega,\rho_n}^2
        +\io\int_{{\Omega^c}} \abs{u(x)}^2\rho_n(x,y)\,dm(x)dm(y)\\
        &\quad-2\io\int_{{\Omega^c}} u(x)u(y)\rho_n(x,y)\,dm(x)dm(y)
        \\&\quad+\io\int_{{\Omega^c}} \abs{u(y)}^2\rho_n(x,y)\,dm(x)dm(y)
        \\&=
        \semin{u}{\Omega,\rho_n}^2
        +\iint_{M\times M}\abs{\chi_\Omega(x)-\chi_\Omega(y)}^2\rho_n(x,y)\,d\mu_{u^2,1}(x,y)\\&\quad
        +\iint_{M\times M}\abs{\chi_\Omega(x)-\chi_\Omega(y)}^2\rho_n(x,y)\,d\mu_{u,u}(x,y)\\&\quad
        +\iint_{M\times M}\abs{\chi_\Omega(x)-\chi_\Omega(y)}^2\rho_n(x,y)\,d\mu_{1,u^2}(x,y).
    \end{split}
    \end{equation}
    We compute the limit of each summand separately.\\
    By applying \cite[Theorem 2.16]{HanPinamontiXu2025} to the metric measure space $(\Omega,d,m_{|\Omega})$ with kernels $\rho_n$, we get 
    \begin{equation}
        \semin{u}{\Omega,\rho_n}^2\to \alpha(\chi_\Omega)\norma{u}{\elleom2}^2\quad\text{as } n\to+\infty.
    \end{equation}
    Then we apply \cite[Theorem 2.16]{HanPinamontiXu2025} to the function $\chi_\Omega$ on the metric measure spaces $(M,d,\mu_{v,w})$ for every choice of $(v,w)\in \{u^+,u^-\}^2\cup \{(1,u^2), (u^2,1)\}$. This implies that
    \begin{align}
        &
        \lim_{n\to+\infty}\iint_{M\times M}\abs{\chi_\Omega(x)-\chi_\Omega(y)}^2\rho_n(x,y)\,d\mu_{u^2,1}=\alpha(\chi_{\Omega^c})\io u^2\,dm\\
        &
        \lim_{n\to+\infty}\iint_{M\times M}\abs{\chi_\Omega(x)-\chi_\Omega(y)}^2\rho_n(x,y)\,d\mu_{u,u}=\alpha(u\chi_{\Omega^c})\io u\,dm\\
        &
        \lim_{n\to+\infty}\iint_{M\times M}\abs{\chi_\Omega(x)-\chi_\Omega(y)}^2\rho_n(x,y)\,d\mu_{1,u^2}=\alpha(u^2\chi_{\Omega^c})m(\Omega)
    \end{align}
    It now suffices to sum all the terms to conclude.
\end{proof}

\begin{remark}
    The proof of this result appears much simpler than the one of Theorem \ref{Theo:MS pointwise:euclidean}. While, in principle, all the terms appearing in the proof of Theorem \ref{Theo:MS pointwise:euclidean} could be treated by applying the result of \cite{HanPinamontiXu2025} to different metric measure spaces, most of the work done in \Cref{Sec:puntuale mms} is aimed at proving some version of Assumption \ref{Ass:chiOmega Sobolev}.
\end{remark}

While \Cref{Theo:MS pointwise:mms} features the term $\alpha(u\chi_{\O^c})$ rather than $\alpha(u)$, the contribution of $u$ restricted to $\Omega$ is negligible for the computation of $\alpha(u)$ under standard assumptions on the kernels $\{\rho_n\}_n$. This behaviour occurs, for example, when considering standard Euclidean kernels (see also \cite[Corollary 2.5]{DaDfGioPi-Necess&Suff}).

\begin{lemma}\label{lem:vanishing_contribution}
    Let $\Omega$ be an open subset of $M$, $v \in L^1(\Omega)$ be a nonnegative function and assume that:
    \begin{enumerate}[label=(\roman*)]
        \item $\lim_{n\to\infty} \rho_n(x,y) = 0$ for $m$-a.e. $x,y \in M$;
        \item there exists a constant $C \ge 0$ such that $\rho_n(x,y) \leq C$ for $m$-a.e.\ $x,y \in M$ with $d(x,y) \geq 1$.
    \end{enumerate}
    Then
    \[
        \alpha(v) = \alpha(v\chi_{\Omega^c}).
    \]
\end{lemma}

\begin{proof}
    By linearity, we have $\alpha(v) = \alpha(v\chi_{\Omega^c}) + \alpha(v\chi_{\Omega})$. Thus, it suffices to show that $\alpha(v\chi_{\Omega}) = 0$.
    Recall that $\alpha(v\chi_{\Omega})$ is defined via the limit of integrals over the domain $\Omega \cap B_1^c(x)$.
    For any $n \in \mathbb{N}$, consider the function $f_n(y) := v(y)\rho_n(x,y)$ restricted to $\{y \in \Omega : d(x,y) \ge 1\}$.
    Assumption (ii) implies the pointwise bound
    \[
        |f_n(y)| \leq C |v(y)| \quad \text{for } m\text{-a.e. } y \in \Omega \cap B_1^c(x).
    \]
    Since $v \in L^1(M)$, the function $g := C|v|$ acts as a dominating function in $L^1$. 
    Moreover, assumption (i) ensures that $f_n(y) \to 0$ pointwise almost everywhere as $n \to \infty$. 
    Therefore, the Dominated Convergence Theorem applies, yielding
    \[
        \lim_{n\to\infty} \int_{\Omega \cap B_1^c(x)} v(y)\rho_n(x,y) \, dm(y) = 0,
    \]
    which concludes the proof.
\end{proof}
We now present an application of \Cref{Theo:MS pointwise:mms} to the setting of Carnot Groups (see also \cite{FracLaplacianCarnot}).

\begin{example}[Asymptotics of the fractional perimeter on Carnot Groups]
    Let $(\mathbb G,\cdot)$ be a Carnot Group of homogeneous dimension $Q$ and let $d_{CC}$ and $\mu$ be, respectively its Carnot-Carath\'eodory distance and the Haar measure. Consider the family of convolution kernels
    \[
    \rho_n(x,y)=\frac{s_n}{d_{CC}(x,y)^{Q+2s_n}}
    \]
    for some sequence $s_n\to0^+$.
    Consider a domain $\Omega\subset G$ such that, for every $t>0$, the $\mathcal{H}^{Q-1}$ measure of the set
    \[
    \{x\in\Omega\mid d_{CC}(x,\Omega^c)=t\}
    \]
    is controlled by a constant depending only on $\Omega$, $Q$ and $r$ (this is the case, for example, for $d_{CC}$-balls of a given radius).
    One can then repeat the proof of \Cref{lemma:product} by means of the coarea formula (see \cite[Theorem 2.4]{BNP-BVfunctionsonPI}, \cite{AdM-Equiv-BV}, \cite{Miranda-BV}) and the fractional Sobolev inequality (see \cite{FracPoincMetric}) to ensure the validity of the assumptions stated in \Cref{Sec:puntuale mms}.
    In this setting, the space $(\mathbb G,d_{CC},\mu)$ and the domain $\Omega$ satisfy the assumption of \Cref{Theo:MS pointwise:mms}.
    In particular, for any set $E\subset \mathbb G$ such that $\alpha(\chi_E)$ exists, one has
    \[
    \lim_{s\to0^+} s\operatorname{Per}_s(E;\Omega)= (Q\mu(B_1)-\alpha(E)) \mathcal{L}^d(E\cap \Omega)+\alpha(E)\mathcal{L}^d(\Omega\setminus E)
    \]
    where $B_1$ is the $d_{CC}$ unit ball.
    
    Note that, in analogy with the Euclidean case,
    \[
    \alpha(E)=\lim_{s\to0^+}\int_{E\cap B_R^c(0)}\frac{s}{\norma{y}{\mathbb{G}}^{Q+2s}}d\mu(y)
    \]
    where $\norma{\cdot}{\mathbb G}$ is a homogeneous norm on $\mathbb{G}$ such that $d_{CC}(x,y)=\norma{x-y}{\mathbb{G}}$. 
\end{example}

We now show that, if the kernels do not concentrate their mass at infinity, namely, if  \Cref{Ass:vanishes_on_bounded} is not fulfilled, \Cref{Theo:MS pointwise:mms} does not hold.\\
\color{magenta}
\normalcolor
\begin{remark}[Asymptotics of the Gaussian fractional perimeter, see \cite{GaussianFracPer}]
\label{Remark:GaussPer}
    Let $(\mathbb{R}^d, d_E, \gamma)$ denote the Euclidean space endowed with the Gaussian measure. Consider the family of Mehler kernels $\{\rho_s\}_{s>0}$ defined by
    \begin{equation}\label{MehlerKernels}
                \rho_s(x,y) = \int_0^\infty \frac{M_t(x,y)}{t^{\frac{s}{2}+1}} \, dt.
    \end{equation}
    where $M_t$ is the Mehler kernel, which has the form
    \[M_t(x,y)
:=
\frac{1}{(1 - e^{-2t})^{N/2}}
\exp\!\left(
-\frac{e^{-2t}\lvert x\rvert^2 - 2e^{-t} x\cdot y + e^{-2t}\lvert y\rvert^2}
{2(1 - e^{-2t})}
\right).
\]
Then, \cite[Main Theorem]{GaussianFracPer} states that
\[
    \lim_{s\to 0} s P^\gamma_s(E, \Omega) = 2 \big[ \gamma(E)\gamma(\Omega\setminus E) + \gamma(E\cap\Omega)\gamma(E^c\cap\Omega^c) \big].
\]

This result does not follow from \Cref{Theo:MS pointwise:mms}. Indeed, the kernels under consideration here do not fulfill \Cref{Ass:vanishes_on_bounded}. In fact, the kernels $\rho_s(x,y)$ converge to $1$ almost-everywhere (see \cite[Lemma 2.6]{CaselliGennaioli} and \cite[Section 4.3]{CaselliGennaioli}). Moreover, since the Mehler kernels decay exponentially as $t\to0^+$, it is possible to apply the dominated convergence theorem in \ref{Eq:ass_mass_at_infinity} to see that the limit is not zero for a generic function in $L^2$. This global behavior, in which the limit depends on the total measure of the sets rather than localizing to the boundary, is a consequence of the finiteness of the reference measure $\gamma$.
Indeed, as shown in \cite[Theorem 1.8]{CaselliGennaioli}, this asymptotic behavior holds for any complete, stochastically complete Riemannian manifold $(M,g,\mu)$ with finite measure $\mu(M)<+\infty$.
The mechanism driving this phenomenon is the long-time behavior of the heat kernel, which converges to $\mu(M)^{-1}$ as $t \to \infty$ (see \cite[Lemma 2.6]{CaselliGennaioli}), contrasting with the infinite volume case where the heat kernel tends to zero and the asymptotics recover a result akin to the one of \cite{Quartet} (see \cite[Theorem 1.6]{CaselliGennaioli}). 
\end{remark}

\section{Acknowledgments}
The work of E.D. and A.F. was supported by the Austrian Science Fund (FWF) through projects  10.55776/Y1292, 10.55776/P35359, and 10.55776/F100800. The work of M.P.  was funded by the European Union - NextGenerationEU, in the framework of the PRIN Project Contemporary perspectives on geometry and gravity (code 2022JJ8KER – CUP G53D23001810006). The views and opinions expressed are solely those of the authors and do not necessarily reflect those of the European Union, nor can the European Union be held responsible for them. M.P. is a member of the GNAMPA group of INdAM.      
\\The authors declare no conflicts of interests.\\
For open access purposes, the authors have applied
a CCBY public copyright license to any accepted manuscript version arising from this
submission.

\end{document}